\documentclass[twoside]{article}
\usepackage{fancyheadings}

\usepackage{graphicx}
\usepackage{latexsym,amsmath,amssymb}
\usepackage{leftidx}
\usepackage{theorem,url}
\newcommand{\sil}[1]{}

\theoremheaderfont{\bfseries\slshape}
\theorembodyfont{\slshape}
\newtheorem{theoy}{Theorem}[section]

\newtheorem{propy}[theoy]{Proposition}

\newtheorem{lemy}[theoy]{Lemma}
\newtheorem{defy}[theoy]{Definition}

\newtheorem{remark}[theoy]{Remark}

\theorembodyfont{\slshape}
\theorembodyfont{\slshape} %{\rmfamily}

\newenvironment{proof}{\textsc{proof:} }{\ensuremath{\Box}}

\numberwithin{equation}{section}
\newcommand{\Ref}[1]{(\ref{#1})}

\newcommand{\theor}[1]{Theorem~\ref{#1}}

\newcommand{\ds}{\displaystyle}
\newcommand{\hb}[1]{\hbox{ #1 }}

\newcommand{\N}{\ensuremath{\mathbb{N}}}

\newcommand{\R}{\ensuremath{\mathbb{R}}}
\newcommand{\Q}{\ensuremath{\mathbb{Q}}}

\newcommand{\dt}[1]{\frac{\hbox{d} #1}{\hbox{d}t  }}

\newcommand{\tr}{\ensuremath{\leftidx{^t}}}

\newcommand{\Gl}{\mathrm{GL}}
\newcommand{\Gle}{\mathfrak{gl}}

\newcommand{\Sp}{\mathrm{Sp}}
\newcommand{\Spe}{\mathfrak{sp}}
\newcommand{\ASp}{\mathcal{S}}
\newcommand{\ASpe}{\mathfrak{s}}

\newcommand{\gva}{\ensuremath{\mathcal{G}} }
\newcommand{\Ge}{\ensuremath{\mathfrak{g}} }

\newcommand{\zva}{\ensuremath{\mathcal{M}} }

\newcommand{\sva}[1][]{\ensuremath{\mathcal{P}^{#1}} }
\newcommand{\inv}{\ensuremath{\bar\iota}}

\newcommand{\XX}[1][]{\ensuremath{X^{#1}} }
\newcommand{\X}[1][0]{\ensuremath{X^{(\!#1\!)}} }

\newcommand{\x}[2]{\ensuremath{x^{(\!#1\!)}_{#2}} }

\newcommand{\jva}[1][]{\ensuremath{{\mathrm{J}}^{#1}}}
\newcommand{\D}[1][]{{\mathrm{D}}^{#1}}

\newcommand{\dd}{\hbox{d}}

\title{Lagrangian Curves in a 4-dimensional affine symplectic space}

\author{ Emilio Musso\thanks{
This research is partially supproted by MIUR (Italy) under the PRIN project
 {\it Varieta' reali e complesse : geometria, topologia e analisi
   armonica}}
\thanks{Author partially supported  by GNSAGA of INDAM.}
\\ Departement of Mathematical Sciences
\\ Politecnico di Torino, Italy \\
\texttt{emilio.musso@polito.it}
\and
Evelyne Hubert
\\ \textsc{galaad}
\\ INRIA M\'editerran\'ee, Sophia Antipolis, France \\
\texttt{evelyne.hubert@inria.fr}
}

\date{November 2013}
\begin{document}

\maketitle

\begin{abstract}
Lagrangian curves in $\R^4$ entertain intriguing relationships with
second order deformation
of plane curves under the special affine group and
null curves in a 3-dimensional Lorentzian space form.
We provide a natural affine
symplectic  frame for Lagrangian curves. It allows us to classify
Lagrangrian curves with constant symplectic curvatures, to construct a class of Lagrangian tori in $\R^4$
 and determine Lagrangian geodesics.\\[6pt]

\textsc{Keywords:}
Symplectic geometry; Lagrangian planes; Differential invariants;
Moving frame; Lie group actions. \\[6pt]

\textsc{Mathematics Subject Classification:} 53A55; 53A15; 53D12 .
%Symbolic Computation.
\end{abstract}

%\MSC
% 53A55; differential invariants
%53A15; affine geometry
% 53D12 Lagrangian Submanifolds
%14L30\sep %Group actions on varieties or schemes (quotients)
%70G65\sep %Symmetries, Lie-group and Lie-algebra methods
%58D19\sep %Group actions and symmetry properties
%53A55\sep %Differential invariants (local theory), geometric objects
%12H05  % Differential algebra

%{\small \tableofcontents}

\newpage

\section*{Introduction}
  \label{intro} %\input{intro}

The study of submanifolds in an affine symplectic space originated in
the work of S.S.~Chern and H.C.~Wang  \cite{CW}. The issue has however
remained silent for many years, before being taken up on several occasions in
 recent literature \cite{AD,D,KOT,MN3,MK,V}. The renewed interest
in this topic raises in connection with the modern approach to the
moving frame method \cite{KOT} and in investigations on integrable
evolutions of curves in affine spaces or in the Grassmannians of the
Lagrangian linear subspaces of $\R^{2n}$ \cite{MB1,MB2,V}. The
specific nature of the geometry of an affine symplectic space stems
from the fact that the linear symplectic group does not act
transitively on the Grassmannians of the linear subspaces of
$\R^{2n}$. Already in the case of curves, the phenomenology is rather
varied and depends on the typology of the osculating spaces along the
curve which can be symplectic, isotropic, coisotropic or
Lagrangian. This makes a non-trivial task the construction of a moving
frame that works effectively in all possible cases, or say for all
 linearly full curve. So far, in the literature
only the generic case (i.e. curves whose osculating spaces of even
order are symplectic) has been investigated, while those cases which
are more specific to the symplectic setting have not been examined.

In the present paper we focus on Lagrangian curves.
They are the curves whose osculating spaces of order $\le n$ are
isotropic. First we provide an appropriate moving frame for those curves.
It was claimed in \cite{KOT} that Lagrangian curves were singular and
required higher order moving frame. Such is not the case. We exhibit a
frame that is of minimal order for generic curves but specializes well
for Lagrangian curves.
Second, with the moving frame at hand, we investigate
 the geometry of some remarkable Lagrangian curves.
We first provide a classification of Lagrangian curves with constant
symplectic curvatures. We then  examine the notion of geodesics for
Lagrangian curves. It turns out to be a subclass of the Lagrangian
curves with constant curvatures.

 In our treatment we
only consider the four dimensional case. This choice is due to two
main reasons: minimize the computational complications when
constructing the moving frame and exploit the specificity of the
four-dimensional case. We shall indeed disclose some interesting interconnections with
the deformation problem \cite{Ca1,Gr,J} for plane curves and with the
conformal geometry of null curves in a $3$-dimensional
pseudo-Riemannian space form (Minkowski, de Sitter or anti-de
Sitter). This latter topic has its own significance in mathematics,
theoretical physics and biophysics
\cite{BCDGM,CGR,In-Lee,KP,MN1,MN2,MN4,N,NFS1,NFS2,NFS3,NMMK}.

The material is organized into five sections.
In Section~\ref{symplectic}
we recall basic facts on the affine symplectic geometry of
$\R^4$ and define the symplectic curvatures.
In Section~\ref{lagrangian} we study Lagrangian
curves. They are characterized by the fact that their
osculating planes are Lagrangian
subspaces of $\R^4$ and demonstrate intriguing connections to conformal
and affine geometry.
Lagrangian curves have a
natural parametrization  through which a notion of
symplectic length can be defined.
We then examine their link to the conformal geometry of the Grassmannian of the
Lagrangian vector subspaces of $\R^4$.
In Section~\ref{techno} we face the construction of a moving frame for linearly
full curves in $\R^4$. We briefly outline the moving frame technology
for parametrized submanifolds following
\cite{fels99,hubert09,hubert12,hubert13}. It is then  applied to build the moving frame
associated to a minimal-order section for the action of the affine
symplectic group on the fourth-order jet space of linearly full
parameterized curves of $\R^4$, regardless of the specific properties
of the osculating spaces. We will also examine the cross-section
corresponding to a Gram-Schmid process. Such a construction can be
extended to provide appropriate moving frame in any dimension.
In Section~\ref{remarkable} we classify Lagrangian curves with constant symplectic
curvatures and we show that, up to affine symplectic transformations,
closed Lagrangian curves with constant curvatures depend on a rational
parameter. Based on those closed curve we show how to construct
Lagrangian tori.
Finally, in Section~\ref{geodesics} we study the symplectic arc-length functional for
Lagrangian curves and we prove that its critical points form a subset
of the Lagrangian curves with constant symplectic curvatures.

\section{Curves in symplectic affine geometry}
  \label{symplectic} %\input{symplectic}

\subsection{Symplectic structure}

On $\R^{4}$ we consider the {\it standard symplectic form}
$$\Lambda(X,Y)=\tr{X}\cdot J\cdot Y = \sum_{a,b=1}^{4}  x_a\, J_{ab}\,y_b,\quad \forall X,Y\in \R^{4},$$
where
$$J=(J_{ab})= \begin{pmatrix} 0 & {I}_{2} \\   - {I}_{2} &
  0       \end{pmatrix},
\quad
X = \begin{pmatrix} x_1 \\ x_2 \\ x_3 \\ x_4 \end{pmatrix},
\quad
Y = \begin{pmatrix} y_1 \\ y_2 \\ y_3 \\ y_4 \end{pmatrix}.$$

The \emph{linear symplectic group}  $\Sp(4,\R) $ is the group of matrices that
preserve $\Lambda$:
$$\Sp(4,\R)=\{\left. A\in \Gl(4,\R) \;\right|\; \tr{A} \,J\, A = J\}.$$
It is a 10 dimensional Lie group.
The  semi-direct product $ \R^4 \rtimes \Sp(4,\R)  $ is the
 {\it affine symplectic  group}.
Its action on $\R^4$ is given by  \( (a,A) \star x = A x +a.\)
It admits the matrix representation:
\begin{equation}\label{ss}\ASp(4,\R) = \left\{ \left.
       \begin{pmatrix}   1 & 0 \\  a & A \\ \end{pmatrix} \in\Gl(5,\R)
   \;\right|\; a\in\R^4, A\in\Sp(4,\R)      \right\}.\end{equation}

The Lie algebras $\Spe(4,\R)$ and  $\ASpe(4,\R)$ of $\Sp(4,\R)$ and
$\ASp(4,\R)$ are respectively
\[ \Spe(4,\R)=\left\{\left. A\in \Gle (4,\R) \;\right|\; \tr{A}\cdot J+J\cdot A=0\right\}\]
and
\[ \ASpe(4,\R) =\left\{  \left.\begin{pmatrix} 0 & 0 \\ a &    A \end{pmatrix} \;\right|\; a \in\R^4,\,
  A\in\Spe(4,\R)\right\}.\]
We can deduce that
\[ \Spe(4,\R)  =\left\{ \left. \begin{pmatrix}  A & B \\ C &-\tr{A} \end{pmatrix}
\;\right|\;
  A,B,C\in \Gle(2,\R) \hb{and} \tr{B}=B, \tr{C}=C\right\}\]
so that a basis for $\Spe(4,\R)$ is provided by the matrices
$$\bar A^i_j= \begin{pmatrix} E_{ij} & 0 \\ 0 & -E_{ji}
              \end{pmatrix},\quad  i,j=1,2,$$
\begin{equation}\label{smat}
\bar B^i_i= \begin{pmatrix}     0 & E_{ii}\\        0 & 0
              \end{pmatrix},\quad
\bar C^i_i=   \begin{pmatrix}   0 & 0 \\ E_{ii}& 0
        \end{pmatrix},\quad i=1,2, \end{equation}
and
$$\bar B^1_2= \begin{pmatrix}     0 & E_{12}+E_{21} \\        0 & 0
              \end{pmatrix},\quad
\bar C^1_2=   \begin{pmatrix}   0 & 0 \\ E_{12}+E_{21} & 0
        \end{pmatrix},
$$
where $E_{ij}$ is the matrix having $1$ at
the position $(i,j)$  as the only nonzero entry.
Let ${A}_i^j$, $B_i^j$ and $C_i^j$ be the corresponding $5\times 5$
matrices in $\ASpe(4)$.
Together with the infinitesimal translations
\begin{equation}\label{tmat} T_1 = \left( \begin{array}{c|cccc} 0 & 0 & 0 & 0 & 0 \\ \hline
       1 & 0 & 0 & 0 & 0 \\ 0  & 0 & 0 & 0 & 0 \\ 0  & 0 & 0 & 0 &  0
      \\ 0  & 0 & 0 & 0 &
      0 \end{array}\right),\; \ldots , \;
T_4 = \left( \begin{array}{c|cccc} 0 & 0 & 0 & 0 & 0 \\  \hline
       0 & 0 & 0 & 0 & 0 \\ 0  & 0 & 0 & 0 & 0 \\ 0  & 0 & 0 & 0 &  0
      \\ 1  & 0 & 0 & 0 & 0\end{array}\right),
\end{equation}
they form a basis for $\mathfrak{s}(4,\R)$.
The corresponding infinitesimal generators of the action of
$\Sp(4,\R)$ on $\R^4$ are the vector fields:
$$\begin{array}{c}
\ds\widetilde{A}_i^j=x_j\frac{\partial}{\partial x_i}-x_{2+i}\frac{\partial}{\partial x_{2+j}},\quad i,j=1,2,\\
\ds\widetilde{B}_i^i=x_{2+i}\frac{\partial}{\partial x_{i}},\quad
      \widetilde{C}_i^i=x_i\frac{\partial}{\partial x_{2+i}},\quad i=1,2\\
\ds \widetilde{B}_1^2=x_{3}\frac{\partial}{\partial  x_{1}}+x_{4}\frac{\partial}{\partial x_{2}},\quad
      \widetilde{C}_1^2=x_1\frac{\partial}{\partial x_{3}}+x_2\frac{\partial}{\partial x_{4}}
\end{array}
$$
and
$$ \widetilde{T}_a=\frac{\partial}{\partial x_a},\quad a=1,...,4 .
$$

\paragraph{Structure in terms of frames:}
A basis $(E_1,E_2,E_{3},E_{4})$ of $\R^{4}$ is said to be symplectic if
$\Lambda(E_a,E_b)=J_{ab}$, for every $a,b=1,...,4$.
This is equivalent to the matrix $\mathbf{E}$ with column
vectors $E_1,E_2,E_3,E_4$ to belong to $\mathrm{Sp}(4,\R)$.

An {\it affine symplectic frame} $(p, {E})$ consists of a point
$p\in \R^{4}$, the origin, and a symplectic basis ${E}$.
The manifold of all these frames can be identified with $\ASp(4,\R)$, the matrix representation
of the affine symplectic group defined as (\ref{ss}). Differentiating the maps
$$
p  : (p,{E})\in \mathcal{S}(4,\R)\to p\in \R^{4},
\quad E_a : (p,{E})\in \mathcal{S}(4,\R)\to E_a\in \R^{4},
$$
and taking into account the identities $\Lambda(E_a,E_b)=J_{ab}$ we find
\begin{equation}\label{MC1}\begin{aligned}& dp= \tau^1E_1+\tau^2E_2+\tau^3E_3+\tau^4E_4,\\
& dE_j= \alpha_j^1 E_1+\alpha_j^2 E_2 + \beta_j ^1 E_{3}+\beta_j ^2 E_{4},\quad j=1,2,\\
& dE_{2+j}= \eta_j^1 E_1+\eta_j^2 E_2 - \alpha_i ^1 E_{3}- \alpha_i ^2 E_{4},\quad j=1,2,\\
\end{aligned}\end{equation}
where $\beta^i_j=\beta^j_i$ and $\eta^i_j=\eta^j_i$ for every $i,j=1,2$.
Note that
\begin{equation}\label{cobasis}\begin{aligned}
 (\tau^1,\tau^2,\tau^3,\tau^4,\alpha^1_1,\alpha^2_2,\alpha^1_2,
  \alpha^2_1,\beta^1_1,\beta^2_2,\beta^1_2,\eta^1_1,\eta^2_2,\eta^1_2)
\end{aligned}\end{equation}
is a basis for the vector space of the left-invariant 1-forms of
$\mathcal{S}(4,\R)$ that is dual to the infinitesimal generators
associated to
\[ \left(  T_1, T_2, T_3, T_4, A_1^1, A^2_2, A^1_2,
   A^2_1, B^1_1, B^2_2, B^1_2, C^1_1, C^2_2, C^1_2\right) \]
when the action of  the one dimensional group
 determined by $\frak{a}\in \ASpe$ is given  by $B\in \ASp \to B e^{t \frak{a}}$.

\paragraph{Structure in terms of the Maurer-Cartan form:}

Another way to organize the structural information is to consider the
left Maurer-Cartan form
\[ \hat\Omega = \begin{pmatrix} 1 & 0 \\ -A^{-1} a & A^{-1}  \end{pmatrix}
\begin{pmatrix} 0 & 0\\ \dd a & \dd  A \end{pmatrix} \]
that is the left invariant form on on $\ASp(4,\R)$ with values in
$\ASpe(4,\R)$.
Hence
\[ \hat\Omega =
\sum_{1\leq i,j\leq 2} \alpha_i^j \, {A}_i^j
+\sum_{1\leq i\leq j \leq 2} \beta_i^j \,{B}_i^j
+\sum_{1\leq i\leq j\leq 2} \eta_{i}^j\, {C}_i^j +
\sum_{a=1}^4\tau_a\,{T}_a
\]
that is
\[ \hat\Omega = \begin{pmatrix} 0 & 0 & 0 & 0 & 0 \\
\tau_1 & \alpha_1^1 & \alpha_1^2 & \beta_1^1 & \beta_1^2 \\
\tau_2 & \alpha_2^1 & \alpha_2^2 & \beta_1^2 & \beta_2^2 \\
\tau_3 & \eta_1^1 & \eta_1^2 & -\alpha_1^1 & -\alpha_2^1 \\
\tau_3 & \eta_1^2 & \eta_2^2 & -\alpha_1^2 & -\alpha_2^2
\end{pmatrix}.\]
The left Maurer-Cartan form satisfies the structure equation
\[ \dd \hat\Omega + \frac{1}{2}\left[\hat\Omega , \hat \Omega\right]=0,
\; \hbox{ that is }\; \dd \hat\Omega + \hat\Omega \wedge \hat\Omega =0\]
in the case of matrix groups.
It is an essential tool for the classification of manifolds.
\begin{theoy} \label{fundamental} \cite[Theorem 5.2 and 6.1]{sharpe97},
\cite[Theorem 1.6.10]{ivey03}
Let $\zva$ be a manifold endowed with a $\ASpe$-valued one-form $\tilde\Omega$
satisfying $\dd \tilde\Omega =-\tilde \Omega \wedge \tilde\Omega$. Then for any point
$z\in\zva$ there exists a neighborhood $U$ of $z$ and a map
$\tilde{\rho}: U \rightarrow \ASp$ s.t. $\tilde\rho^* \Omega = \tilde\Omega$. Any
two such maps $\tilde{\rho}_1$, $\tilde{\rho}_2$ satisfies
$\tilde\rho_1=g\cdot \tilde\rho_2$ for some fixed $g\in\ASp$.
\end{theoy}

The proof proceeds by showing the existence of a solution
 $\tilde\rho: \zva \rightarrow \ASp$ to the
differential system $\dd \tilde\rho = \tilde{\rho}\, \tilde\Omega$.

We will alternatively use the right Maurer-Cartan form
\[ \Omega =
\begin{pmatrix} 0 & 0\\ \dd a & \dd  A \end{pmatrix}
 \begin{pmatrix} 1 & 0 \\ -A^{-1} a & A^{-1}  \end{pmatrix}
\]
that satisfies the structure equation
$\dd \Omega = \Omega \wedge \Omega$.
This is motivated by the fact that the infinitesimal generators are
implicitly defined with a right invariant structure on the group \cite{olver:yellow}.

\subsection{Symplectic curvatures}

Let $\jva[k](\R,\R^{4})$ be the space of $k^{th}$-order jets of smooth
parameterized curves $\gamma : \R\to \R^4$,
with coordinates $\X[0]$, $\X[1],\dots, \X[k]$.
We can represent a point in $\jva[k]$ by  an ordered pair $(t,\mathrm{X})$, consisting of the independent variable $t\in \R$ and a $5\times (k+1)$ matrix
\[ \mathrm{X}=\begin{pmatrix}
     1 & 0 & \ldots & 0 \\
     \X[0] & \X[1] & \ldots & \X[k]
\end{pmatrix}\]
so that the action of  $\mathcal{S}(4,\R)$ is given by matrix
multiplication:
\[  \begin{pmatrix} 1 & 0 \\ a & A \end{pmatrix}
\begin{pmatrix}
     1 & 0 & \ldots & 0 \\
     \X[0] & \X[1] & \ldots & \X[k]
\end{pmatrix}
=
\begin{pmatrix}
     1 & 0 & \ldots & 0 \\
    A \X[0] +a & A\X[1] & \ldots & A \X[k]
\end{pmatrix}\]
% The infinitesimal generators associated to the above described
% basis of $\ASpe(4,\R)$ are the vector fields:
% \[  %\label{INFGEN}
% \begin{array}{c}
% \ds A_i^j=\sum_{\ell=0}^{k}\left(\x{\ell}{j}\frac{\partial}{\partial \x{\ell}{i}}-\x{\ell}{2+i}\frac{\partial}{\partial \x{\ell}{2+j}}\right),\quad i,j=1,2,\\
% \ds B_i^i=\sum_{\ell=0}^{k}\left(\x{\ell}{2+i}\frac{\partial}{\partial
%     \x{\ell}{i}}\right),\quad
% C_i^i=\sum_{\ell=0}^{k}\left(\x{\ell}{i}\frac{\partial}{\partial \x{\ell}{2+i}}\right),\quad i=1,2,\\
% \ds B_1^2=\sum_{\ell=0}^{k}\left( \x{\ell}{3}\frac{\partial}{\partial
%     \x{\ell}{2}}+\x{\ell}{4}\frac{\partial}{\partial
%     \x{\ell}{1}}\right),\quad
% C_1^2=\sum_{\ell=0}^{k} \left(\x{\ell}{1}\frac{\partial}{\partial
%     \x{\ell}{4}}+\x{\ell}{2}\frac{\partial}{\partial \x{\ell}{3}}\right),\\
% \ds T_a=\sum_{\ell=0}^{k} \frac{\partial}{\partial \x{\ell}{a}},\quad
%  a=1,...,4.
% \end{array}
% \]

The maps
$$\Lambda(\X[i],\X[j]) =-\x{i}{3}\x{j}{1}-\x{i}{4}\x{j}{2}+\x{i}{1}\x{j}{3}+\x{i}{2}\x{j}{4}$$
are differential invariants for the action of the affine symplectic group on the jet space $\jva[k]$.
We single out the \emph{symplectic curvatures} $\kappa_i = \Lambda\left(\X[i], \X[i+1]\right)$, for
$i\geq 1$.
Any other differential invariant can be expressed in terms of those. For instance:
\[
\Lambda(\X[1], \X[2])=\kappa_1,\quad
\Lambda(\X[1], \X[3]) = \kappa_1',\quad
\Lambda(\X[1],\X[4]) = \kappa_1''-\kappa_2,
\]\[
\Lambda(\X[1],\X[5]) = \kappa_1'''-2\kappa_2',\quad \Lambda(\X[2], \X[3]) = \kappa_2,
\quad  \Lambda(\X[2], \X[4]) = \kappa_2' ,
\]\[
\Lambda(\X[1], \X[5]) = \kappa_2'-\kappa_3,\quad \Lambda({\X[3]}, \X[4]) = \kappa_3 , \quad \Lambda(\X[3],\X[5]) = \kappa_3' . \]
We shall come across another differential invariant
\begin{equation} \label{phi}
\phi = \det(\X[1],\X[2],\X[3],\X[4])
=\kappa_2^2-\kappa_1\kappa_3+\kappa_1'\kappa_2'-\kappa_2\kappa_1''
.
\end{equation}

\cite{MB2,KOT,V} considered curves $\gamma: I \rightarrow \R^{4}$
s.t. $\kappa_1\circ j^h(\gamma) = \Lambda(\gamma',\,\gamma'') $  does not vanish anywhere on
$I$, where $j^h(\gamma)$ denotes the canonical lifting of $\gamma$ to any of the jet spaces $\jva[h]$, $h\ge 2$.
% Introducing the \emph{symplectic arc-length}
% \[ s(t) = \int^t \kappa_1^{\frac{1}{3}}\,\dd t \]
% they provided Serret-Frenet equations for curves under the action of
% $\ASp(4,\R)$.
The present article focuses on the curves where $\kappa_1$ is
identically zero.

\section{Lagrangian curves}
    \label{lagrangian} %\input{lagrangian}

In this section we define Lagrangian curves and show some of their
properties.
We make explicit their relevance to conformal
and affine geometry.
We first exhibit a natural parametrization based on which we can
define a symplectic arc length. We briefly discuss the standard conformal structure of the Grassmannian of
the oriented Lagrangian planes of $\R^4$ (we refer to
\cite{BCDGM},\cite{GSkepler} for more details). Subsequently we define
the osculating curve and the phase portraits of a Lagrangian curve. We
prove that the osculating curve is a null-curve in the Grassmannian of
Lagrangian planes and we observe that any such curve satisfying some
mild conditions arises in this fashion. This fact highlights the links between Lagrangian curves
in $\R^4$ with null curves in a $3$-dimensional Lorentzian space form.
Secondly, we characterize
Lagrangian curves in terms of the behavior of the phase portraits.
We prove that the phase portraits of a Lagrangian curve are second
order deformations of each other with respect to the action of the
special affine group. The converse is also true, provided that the two
phase portraits do not have inflection points.

\subsection{Definition and parametrisation}

\begin{defy}
   A smooth parameterized curve $\gamma : I\to \R^{4}$ is
    said to be {\it Lagrangian} if $\gamma'(t)$ and
    $\gamma''(t)$ are linearly independent and
    $\Lambda (\gamma'(t),\gamma''(t))=0$.
\end{defy}

The definition does not depend on the parametrization,
    and  is invariant by symplectic transformations and
   dilations.
Indeed,
    if $h:J\to I$ is a change of parameters,
    $\Phi:\R^4\to \R^4$ is a symplectic transformation,
    $r$ is a non-zero real number and $\gamma:I\to \R^4$ is a Lagrangian curve,
    then $\widetilde{\gamma}: t \in J \to r\Phi[\gamma(h(t))]\in
    \R^4$ is another Lagrangian curve. More generally, if $\gamma:I\to
    \R^4$ is a Lagrangian curve and if $r:I\to \R$ is a nowhere
    vanishing smooth function, then integrating $r \gamma'$ we
    obtain another Lagrangian curve $\tilde{\gamma}:I\to \R^4$ such
    that the tangent lines of $\gamma$ and $\tilde{\gamma}$ at
    $\gamma(t)$ and $\tilde{\gamma}(t)$ are parallel to each other, for
    every $t\in I$.

\begin{defy}
A Lagrangian curve is {\it non-degenerate} if
$\Lambda(\gamma''(t),\gamma'''(t))\neq 0$,
for every $t\in I$ and {\it linearly full} if
$\gamma'|_t\wedge ... \wedge \gamma^{(iv)}|_t\neq 0$,
 for every $t\in I$.
\end{defy}

With (\ref{phi}) we see that a linearly full Lagrangian curve is automatically non-degenerate.

If $\gamma$ is Lagrangian and non-degenerate, then the differential $1$-form
$$\sigma_{\gamma} = \Lambda(\gamma'',\gamma''')^{1/5}dt=
(\tr\gamma''\, J\, \gamma''')^{1/5}dt,$$ is nowhere
vanishing. Furthermore, it is independent of the parametrization and
invariant by the action of the affine symplectic group. Therefore, we
can define an intrinsic orientation on a non-degenerate Lagrangian
curve requiring that $\Lambda(\gamma''|_t,\gamma'''|_t)>0$, for every
$t\in I$. From now on a Lagrangian curve is equipped with
the intrinsic orientation.

\begin{defy}Let $\gamma$ be a non-degenerate Lagrangian curve, the
    differential form $\sigma_{\gamma}$ is said to be the {\it
      symplectic arc-element} of $\gamma$.
If
    $\Lambda(\gamma'',\gamma''')^{1/5}=1$ (i.e. if
    $\sigma_{\gamma}=dt$), $\gamma$ is said to be  {\it parameterized by the
      symplectic arc-length}.
%  a {\it natural    parametrization} or that $\gamma$ is
\end{defy}

The symplectic arc-length parameter will be denoted by $s$.
Obviously, the symplectic arc-length parametrization
   is unique up to a shift $s\to s+s_0$ of the
   parameter.
If $[a,b]\subset I$ is a closed interval,    the quantity
$$ \int_a^b \sigma_{\gamma}$$
is the {\it symplectic length} of the Lagrangian arc $\gamma([a,b])$. Using standard arguments one can easily prove the following Proposition :

\begin{propy}
    If $\gamma:I\to \R^4$ is a non-degenerate Lagrangian
    curve equipped with its intrinsic orientation, then there exists a
    strictly increasing surjective map $h:J\subset \R\to I$ such that
    $\gamma\circ h$ is a parametrization by
      symplectic arc-length.
\end{propy}

\subsection{Osculating curves}

Here we discuss the links between Lagrangian curves in $\R^4$ and null
curves in a $3$-dimensional Lorentzian space form. Since the notion of
null curve is invariant by conformal transformations, the natural
environment is the conformal compactification of the Minkowski
$3$-space, which can be thought of as the manifold of all oriented
Lagrangian vector subspaces of $\R^4$. Such a link between symplectic
and conformal geometry is specific to the four-dimensional case.The
reason lies in the fact that $\mathrm{Sp}(4,\R)$ is a covering group
of the connected component of the identity of $\mathrm{O}(3,2)$. We
begin with a brief description of the conformal structure of the
Grassmannian of oriented Lagrangian planes in $\R^4$. Then we
introduce the concept of osculating curve by which we establish the
correlations among Lagrangian and null curves.

\begin{defy} An oriented {\it Lagrangian plane} is a two dimensional linear subspace $L
\subset \R^4$ such that $\Lambda_{| {L}}=0$. The Grassmannian $\Lambda_+^2$ of all such planes
is a smooth manifold diffeomorphic to $S^2\times S^1$ (see \cite{BCDGM,MN3}).\end{defy}

The action of $\Sp(4,\R)$ on $\R^4$ induces an action on $\Lambda^2_+$
which is transitive. The projection map
\begin{equation}\label{projection}
  \pi_{\Lambda}:\mathbf{E}\in \Sp(4,\R)\to [E_1\wedge E_2]\in
  \Lambda^2_+
    \end{equation}
makes $\Sp(4,\R)$ into a principal fiber bundle
with structure group
$$
\Sp(4,\R)_1 = \left\{ \mathbf{X}(A,b)=\left(
      \begin{array}{cc}
         A & A\cdot b \\
           0 & ^t A^{-1} \\
            \end{array}
             \right) \, : \,  \det A > 0,\, b\in \mathrm{S}(2,\R) \right\}.
$$
From this, it follows that the $1$-forms
$(\eta^1_1,\eta_2^2,\eta^2_1)$ span the semi-basic
forms\footnote{An
exterior differential form is semi-basic if it annihilates the vertical vectors of the fibration} of the
projection $\pi_{\Lambda}$. Moreover, the symmetric quadratic form
${g}=-\eta^1_1\eta^2_2+(\eta^2_1)^2$ and the exterior 3-form
$\eta^1_1\wedge \eta^2_1\wedge \eta^2_2$ are well defined on
$\Lambda^2_+$, up to a positive multiple. They determine a conformal
structure of signature $(2,1)$ and an orientation, respectively.

\begin{remark} The 3-dimensional Minkowski space $\R^{2,1}$ can be
  identified with the vector space $\mathrm{Q}(2,\R)$ of $2\times 2$
  symmetric matrices equipped with the non-degenerate inner product of
  signature $(2,1)$ induced by the quadratic form
$$B\in \mathrm{Q}(2,\R) \to -\mathrm{det}(B)\in \R.$$
For each $B\in \mathrm{Q}(2,\R)$ we associate the oriented Lagrangian
plane $L(B)$ spanned by the vectors $(1,0,B_1^1,B_1^2)$ and
$(0,1,B_1^2,B_2^2)$. It is easy to check that the map
$$B\in \mathrm{Q}(2,\R)\to L(B)\in \Lambda^2_+$$
is a conformal embedding. This shows that $\Lambda^2_+$ can be viewed as the {\it
conformal compactification} of the
$3$-dimensional Minkowski space (cf. \cite{BCDGM},\cite{GSkepler}).
\end{remark}

\begin{defy}
The osculating spaces to a non-degenerate Lagrangian
    curve $\gamma$ define a smooth map
$$\delta_{\gamma}:t\in I \to [\gamma'(t)\wedge \gamma''(t)]\in \Lambda^2_+$$
to the manifold of the oriented Lagrangian planes. We say that
$\delta$ is the {\it osculating curve} of $\gamma$.
\end{defy}

\begin{propy}{Two non-degenerate Lagrangian curves $\gamma$ and $\tilde{\gamma}$
have the same osculating curve if and only if they have parallel
tangent lines.}\end{propy}
\begin{proof}{If $\gamma,\tilde{\gamma}:I\to \R^4$ have the same osculating curves, then
$\gamma' = r_1\tilde{\gamma}'+r_2\tilde{\gamma}''$, where $r_1,r_2:I\to \R$ are suitable smooth functions. Then, the acceleration of $\gamma$ is given by
$$\gamma''=r_1'\tilde{\gamma}'+(r_1+r_2')\tilde{\gamma}''+r_2\tilde{\gamma}'''.$$
Differentiating $\Lambda(\tilde{\gamma}',\tilde{\gamma}'')=0$ we get $\Lambda(\tilde{\gamma}',\tilde{\gamma}''')=0$. This implies
$$0=\Lambda(\gamma',\gamma'')=r_2^2\Lambda(\tilde{\gamma}'',\tilde{\gamma}''').$$
Bearing in mind that $\tilde{\gamma}$ is non-degenerate, it follows that $r_2=0$. Therefore $\gamma' = r_1\tilde{\gamma}'$, i.e. $\gamma$ and $\tilde{\gamma}$ have parallel tangent lines. Conversely, if $\gamma$ and $\tilde{\gamma}$ have parallel tangent lines, then $\gamma'=r\tilde{\gamma}'$, where $r:I\to \R$ is a smooth function everywhere different from zero. Therefore we have $\gamma'\wedge \gamma''= r^2 \tilde{\gamma}'\wedge \tilde{\gamma}''$, from which it follows that $\delta_{\gamma}=\delta_{\tilde{\gamma}}$.}\end{proof}

\begin{propy}{Let $\gamma:I\to \R^4$ be a non-degenerate Lagrangian curve. Then, $\delta_{\gamma}$ is a null-curve\footnote{ A smooth immersed curve $\delta :I \to \Lambda_+^2$ is {\it null} if its tangent vectors are isotropic (null) with respect to the conformal structure of $\Lambda_+^2$} of $\Lambda^2_+$.}\end{propy}
\begin{proof}{To check if $\delta:I\to \Lambda^2_+$ is a null-curve we choose a lift of $\delta$ to $\mathrm{Sp}(4,\R)$, i.e. any map $\mathbf{E}=(E_1,E_2,E_3,E_4):I \to \mathrm{Sp}(4,\R)$ such that $\delta=[E_1\wedge E_2]$. Subsequently we compute the pull-back of the Mauer-Cartan forms $\eta^1_1,\eta^2_2,\eta^2_1$ and we write $\mathbf{E}^*(\eta^i_j)=c^i_jdt$, where $c^i_j:I\to \R$ are smooth functions. Then, $\delta$ is a null-curve if and only if $c^1_1c^2_2-(c^2_1)^2=0$ and $(c^1_1)^2+(c^2_2)+(c^2_1)^2$ is nowhere vanishing.
\noindent If $\gamma$ is a non-degenerate Lagrangian curve, we consider a {\it second order moving frame} along $\gamma$, that is a smooth map
$$\mathbf{E}=(E_1,E_2,E_3,E_4) :I\to \mathrm{Sp}(4,\R)$$
such that $\gamma'(t)=E_1(t)$ and $\gamma''(t)=E_2(t)$ for every $t\in I$. We then have
$$\left(
    \begin{array}{cc}
      \mathbf{E}^*(\eta^1_1) & \mathbf{E}^*(\eta^2_1) \\
      \mathbf{E}^*(\eta^2_1) & \mathbf{E}*(\eta^2_2)\\
    \end{array}
  \right)=\left(
            \begin{array}{cc}
              0 & 0 \\
              0 & c^2_2 \\
            \end{array}
          \right)dt,$$
where $c^2_2$ is a nowhere vanishing smooth function. On the other hand $\mathbf{E}$ is a lift of the osculating curve $\delta_{\gamma}$ satisfying
$$c^1_1c^2_2-(c^2_1)^2=0,\quad (c^1_1)^2+(c^2_2)+(c^2_1)^2=c_2^2>0.$$ This yields the required result.}\end{proof}

It is furthermore not difficult to prove that, under a mild generic
condition, any null curve of $\Lambda_+^2$ arises as the osculating
curve of a non-degenerate Lagrangian curve of $\R^4$.

\subsection{Phase portraits} \label{dpairs}

We wish to show here the relation between Lagrangian curves and second
order deformation of plane curves under the special affine group.
We first recall  the classical notion of deformation of a plane curve, with
fixed parametrisation,  with respect to a group of transformations.

\begin{defy}
Let $\gva$ be a Lie group and
$(g,\mathbf{x})\in \gva\times \R^2\to g\star \mathbf{x}\in \R^2$
be a (left) action of $\gva$ on $\R^2$.
Two plane curves $\mathbf{a}, \mathbf{b}:I \to \R^2$ are said to be
{\it $k$-th order deformations}
of each other with respect to $\gva$ if there exists a smooth map
$g : I\to \gva$ such that $\mathbf{a}$ and $g(t)\star \mathbf{b}$
have the same $k$-th order jets at $t$, for every $t\in I$.
\end{defy}

If $\gva$ is the Euclidean group of rigid motions,
two curves are first order deformations each other if and only if they
have the same speed,
while second order deformation implies the congruence of the two curves
We shall consider the larger  {\it special affine group}, i.e. the
semi-direct product of special linear group $\mathrm{SL}(2,\R)$
with the group of the translations.

\begin{defy}
The \emph{phase portraits} of a
 curve $\gamma:I\to \R^4$
are the  plane curves
$\mathbf{a}_{\gamma},\,\mathbf{b}_{\gamma}: \, I\to \R^2$
defined by
$\mathbf{a}_{\gamma} = (\gamma_1,\gamma_3)$,
$\mathbf{b}_{\gamma}=(-\gamma_2,\gamma_4)$.

Conversely, given two plane curve $\mathbf{a},\mathbf{b}: I\to \R^2$ we
denote by $\gamma_{(\mathbf{a},\mathbf{b})}:I\to \R^4$  the curve
defined by  $\gamma_{(\mathbf{a},\mathbf{b})} = (a_1,-b_1,a_2,b_2)$.
\end{defy}

A curve $\gamma_{(\mathbf{a},\mathbf{b})}$ is Lagrangian if and only
if  its phase portraits
satisfy  $\|\mathbf{a}'\|+\|\mathbf{b}'\|>0$ and
$\mathbf{a}'\wedge \mathbf{a}''=\mathbf{b}'\wedge \mathbf{b}''$.

\begin{propy}{Let $\mathbf{a}, \mathbf{b}:I \to \R^2$ be two regular
    plane curves without inflection points.
Then, they are second order deformations
each other with respect to the special affine group
if and only if $\gamma_{(\mathbf{a},\mathbf{b})}$ is a Lagrangian
curve of $\R^4$.}
\end{propy}

\begin{proof}
It suffices to prove that $\mathbf{a}$ and $\mathbf{b}$ are second order deformations each other with respect to the special affine group if and only if $\mathbf{a}'\wedge \mathbf{a}''=\mathbf{b}'\wedge \mathbf{b}''$. If $\mathbf{a}$ and $\mathbf{b}$ are second-order deformations each other, then there exist smooth maps $A:I\to \mathrm{SL}(2,\R)$ and $T:I\to \R^2$ such that $A(t_0)\mathbf{a}+T(t_0)$ and $\mathbf{b}$ have the same second order jet at $t_0$, for every $t_0\in I$. This implies
$$\mathbf{b}'|_{t_0}=A(t_0)\mathbf{a}'|_{t_0},\quad \mathbf{b}''|_{t_0}=A(t_0)\mathbf{a}''|_{t_0}.$$
We then have
$$\mathbf{b}'|_{t_0}\wedge \mathbf{b}''|_{t_0}
= \mathrm{det}(A(t_0))\mathbf{a}'|_{t_0}\wedge \mathbf{a}''|_{t_0}
=\mathbf{a}'|_{t_0}\wedge \mathbf{a}''|_{t_0}, \quad \forall t_0\in I.$$
This proves that $\gamma_{(\mathbf{a},\mathbf{b})}$ is a Lagrangian
curve. Conversely, suppose that $\mathbf{a}$ and $\mathbf{b}$ are two
curves without inflection points and that
$\gamma_{(\mathbf{a},\mathbf{b})}$  is a Lagrangian curve. Since the
two curves do not have inflection points the functions
$\mathrm{det}(\mathbf{a}',\mathbf{a}'')$ and
$\mathrm{det}(\mathbf{b}',\mathbf{b}'')$ are nowhere
vanishing. Moreover, the assumption that
$\gamma_{(\mathbf{a},\mathbf{b})}$  is Lagrangian implies
$\mathrm{det}(\mathbf{a}',\mathbf{a}'')
=\mathrm{det}(\mathbf{b}',\mathbf{b}'')$.
Therefore,
if we set $A=(\mathbf{b}',\mathbf{b}'')\cdot
(\mathbf{a}',\mathbf{a}'')^{-1}$ we have a map with values in the
special linear group.
 If we put
$T = \mathbf{b}-A\cdot \mathbf{a}$, we
obtain
$$\mathbf{b}(t_0)
=A(t_0)\mathbf{a}(t_0)+T(t_0),\quad\mathbf{b}'|_{t_0}
=A(t_0)\mathbf{a}'|_{t_0},\quad\mathbf{b}''|_{t_0}=A(t_0)\mathbf{a}''|_{t_0}.$$
This means that $\mathbf{b}$ and $A(t_0)\mathbf{a}+T(t_0)$ have the
same second order jet at $t_0$, for every $t_0\in  I$. So,
$\mathbf{a}$ and $\mathbf{b}$ are second order deformations each other
with respect to the special affine group.
\end{proof}

\section{Symplectic moving frames}
  \label{techno} %\input{techno}

The goal is to determine a $\ASpe(4,\R)$-valued one form $\tilde{\Omega}$ on
jets of Lagrangian curves that satisfies $\dd \tilde\Omega = \tilde\Omega \wedge
\tilde\Omega$.
As an intermediate tool we define, following \cite{fels99}, a
\emph{moving frame} as an appropriate (right) equivariant map
$\rho : \jva(\R,\R^4) \rightarrow \mathcal{S}(4,\R)$.
Then $\tilde \Omega= \rho^* \Omega$, where $\Omega$ is the (right)
Maurer-Cartan form on $\ASp(4,\R)$

A moving frame is implicitly defined by a choice of
\emph{cross-section}. The technology introduced in
\cite{fels99}, and further developed in
\cite{hubert09,hubert13,hubert12,mansfield10},
lends itself to an algorithmic treatment.
Neither the moving frame, nor the differential invariants, need to be
known explicitly to characterize $\tilde{\Omega}$.
In the present case of curves in affine  symplectic geometry, we can
express everything in terms of the symplectic curvatures
that we introduced.

In the first subsection we shortly review the results that enable us
to perform the computations formally. In each of the three next
subsections we examine the results obtained for different choices of
cross-section.
The computations are lead with the set of Maple routines \textsc{aida}
\cite{aida} that works on top of the libraries
\emph{DifferentialGeometry} and \emph{diffalg} \cite{diffalg,ncdiffalg}.
 The first cross-section is the one used in
\cite{MB2,mansfield10,KOT,V}. It assumes that the first symplectic
curvature does not vanish. Contrary to expectations, we
provide a minimal order cross-section that removes this restriction.
As a more general construction for $\ASp(2n,\R)$ we can construct a
moving frame thanks to a symplectic Gram-Schmidt process. This is
illustrated in the last subsection.

\subsection{Moving frames from sections}

For this review, we place ourself in a slightly more general context.
We consider a $r$ dimensional (matrix) Lie group $\gva$
acting on (an open set of)  $\R^n$.
To each element $\frak{a}$ in the Lie algebra $\Ge$ of $\gva$
we can associate a vector field
$V_{\frak{a}}$ the flow of which is an orbit of the action of a one-dimensional
subgroup of $\gva$ classically denoted $e^{\frak{a}t}$.
To a basis of the Lie algebra thus correspond $r$
\emph{infinitesimal generators} of the action
of $\gva$ on $\R^n$.

We consider the jets $\jva[k](\R^m,\R^n)$, or simply $\jva[k]$, of  parameterized
$m$-dimensional submanifolds.
$\D_1,\ldots,\D_m$ are the total derivations with respect to the parameters.
The action of $\gva$ is prolonged to those jet spaces so as to be
compatible with those total derivations. Explicit
prolongation formulae for the action and the infinitesimal generators can be found
in \cite{olver:yellow} for instance. Like many such operations, their implementation
is available through the Maple library \emph{DifferentialGeometry}. In
the following,
$V_1,\ldots,V_r$  denote the appropriate prolongation of the
 infinitesimal generators.

The prolonged action of $\gva$ on $\jva[k]$ is  denoted by
$\star$. It is given by a smooth map
$$\begin{array}{ccc}\gva\times \jva[k]  & \rightarrow & \jva[k] \\
  (\lambda, z) & \mapsto & \lambda \star z \end{array}$$
If the group $\gva$ acts locally effectively, a  mild hypothesis,
there exists $s\in \N$ such that the generic orbits of the prolonged action of
$\gva$  on $\jva[s]$ have the same dimension $r$
as the group \cite[Theorem 5.11]{olver:purple}.
We place ourselves in a neighborhood of a
point $z_0\in \jva[s]$ where the distribution defined by the
prolonged infinitesimal generators $V_1,\ldots,V_r$
of the action has full rank $r$. The orbits of the points in this neighborhood are
of dimension $r$  and the action is locally free there.
Through $z_0$ we can  find a \emph{cross-section}, i.e. a manifold $\mathcal{P}$
of codimension $r$ that  is transverse to the orbits.

Assume the cross-section $\mathcal{P}$ is determined as the level set
$C=(c_1,\ldots, c_r)\in\R^r $ of a map $P=(p_1,\ldots, p_r) : \jva[s]
\rightarrow \R^r$. In other words $\mathcal{P}$ is
defined by the equations $p_1(z)=c_1, \ldots, p_r(z)=c_r$.
Then $\mathcal{P}$ is transverse to the orbits in the neighborhood of
one of its point $z $ if   the  $r\times r$ matrix
\[ V(P) = \left( V_i(p_j) \right)_{i,j} \]
is invertible when evaluated at $z\in\sva$.
Note that the matrix $V(P)$ is the Jacobian of the map
$\gva \rightarrow \R^r$ defined by $\lambda \mapsto
P(\lambda\star z)$ at identity.
By virtue of the implicit function theorem,
there exists a neighborhood $\mathcal{U}$  of $z$
and  a unique smooth map $\rho : \mathcal{U} \rightarrow \gva$ such that
\[ P( \rho(z) \star z) = C \hbox{ and } \rho|_\mathcal{P} = e. \]
This map has the sought equivariant property:
\( \rho(\lambda \star z) = \rho(z) \cdot \lambda^{-1} .\)

Beside the moving frame, a local cross-section allows us to define an
\emph{invariantization process} and the \emph{normalized invariants}\footnote{
Contrary to the moving frame construction, the invariantization
does not restrict to locally free actions. See \cite{hubert07b}.}.
Given a smooth function $f:\mathcal{U}\subset \jva[s+k] \rightarrow \R$ its invariantization
$\inv f: \mathcal{U}\subset \jva[s+k] \rightarrow \R$ is defined by $\inv f( z) = f(\bar{z})$
where $\bar{z}$ is the intersection of the connected part
of the orbit of $z$ with the cross-section $\mathcal{P}$.
Analytically this is given as $\inv f(z) = f(\rho(z)\star z)$.
The \emph{normalized  invariants}
are the invariantization of the coordinate functions \cite{fels99,hubert07b}.
We can compute them algebraically \cite{hubert07b} but it is often preferable
to work with those formally.
This is made possible by the fact that
$\inv f( z) = f(\inv z)$ and $P(\inv z) =C$.
In particular, if $f$ is an invariant then $f(z)=f(\inv z)$.
We can therefore work formally with the normalized invariants
$\inv z$ subjected  to the relationships defined by the chosen
cross-section. This idea is reinforced by
the explicit relation between derivation and invariantization
\cite[Section 13]{fels99}, \cite[Theorem 3.6]{hubert09}:
\begin{equation} \label{rec}
\D(\inv f) = \inv( \D f) - K \cdot \inv \left(V(f)\right) \end{equation}
where $\D=\tr{\begin{pmatrix} \D_1 & \ldots & \D_m\end{pmatrix}}$ and $K$ is the $m\times r$ matrix
\begin{equation} \label{mc}
K
=  \inv \left( \D(P) \right)\cdot \inv\left(  V(P) \right)^{-1}\!\!
\end{equation}
defined with the $m\times r$  and $r\times r$ matrices
\[
\D(P) =\left(\D_i(p_j)\right)_{ 1\leq i\leq m ,\\1\leq j\leq r},
\quad
V(P) = \left( V_i(p_j) \right)_{1\leq i, j\leq r}.
\]
Thanks to this formula we can characterize finite sets of differential
invariants as generating. The first such set is the set of
\emph{normalized invariants} of order $s+1$ and less.
Their complete syzygies are described
 in \cite{hubert09}; They are  built on \Ref{rec}. Of relevance to the
 geometric applications, the Maurer-Cartan invariants are the entries
 of the matrix $K$. From the formula above, we see that those entries
 consist of differential invariants of order $s+1$ at most.
 They form a  generating set of differential invariants \cite{hubert13}.

Formula \Ref{mc} provides the expression of the Maurer-Cartan
invariants in terms of the normalized invariants of order $s+1$.
Conversely,  normalized invariants, and thus any differential invariants, can be written effectively in terms
of Maurer-Cartan invariants thanks to \Ref{rec}.
The syzygies on a set of generating differential invariants allows to
determine smaller sets of generators algorithmically, with
differential elimination \cite{hubert05,hubert09,hubert07c}.

The geometric importance of the Maurer-Cartan invariants comes from
the fact that they describe the pullback by $\rho$ of the (right) Maurer-Cartan
form ${\Omega}$:
\begin{equation}\label{serret}
\rho^* \Omega \equiv - \sum_{i=1}^m\left(\sum_{j=1}^r K_{ij} \, \frak{a}_j \right) dt_i
\end{equation}
 where
 $ \frak{a}_1,\ldots, \frak{a}_r$ form a basis of the Lie algebra of
 $\gva$ and $(dt_1,\ldots,dt_m)$ form the horizontal  one-form basis
 dual to $(\D_1, \ldots, \D_m)$, i.e. $t_1,\ldots,t_m$ are the parameters.
In (\ref{serret}) the symbol $\equiv$ stands for equality modulo the contact ideal.
The Maurer-Cartan invariants $K_{ij}$ are subjected
 to the syzygies determined by the pullback by $\rho$ of the structure equation
$\dd {\Omega} = \Omega \wedge \Omega$.

A  difficulty in the presented formalism is to determine the
neighborhood $\mathcal{U}$ in which the normalized invariants and the
moving frame are well defined. Even for linear actions one quickly encounters
algebraic functions. As we shall see in the next subsections,
the case of the affine symplectic action is particularly well
behave. One already saw that the symplectic curvatures are polynomial
functions. One shall discover that, for a relevant choice of
cross-section, the moving frame and the normalized invariants are
given by rational functions at worst. The open set $\mathcal{U}$ is a
subset of the linearly  full curves determined by a single equation,
either $\kappa_1\neq 0$ for the first choice of cross-section, or
$\kappa_2\neq 0$ for the second and third choice.

\subsection{Section previously considered}
     \label{previous} %\input{previous}

The orbits of the action of $\ASp(4,\R)$ prolonged to  $\jva[4](\R,\R^4)$ are 14
dimensional, as can be checked by computing the rank of the prolonged
infinitesimal generators.
The cross-section chosen in \cite[Section 3.2]{MB2},
\cite[Example 5.5.2]{mansfield10}, \cite{KOT,V}
 is given by the equations:
\[          \x{0}{1}=0,\;\x{0}{2}=0,\;\x{0}{3}=0,\;\x{0}{4}=0 ,\]
\[ \x{1}{1}=1,\, \x{1}{2}=0,\, \x{1}{3}=0,\x{1}{4}=0, \]
\[ \x{2}{1}=0,\, \x{2}{2}=0,\, \x{2}{4}=0,\]
\[  \x{3}{2}+\x{2}{3}=0,\,   \x{3}{4}=0,\]
\[ \x{4}{2}+2\,\x{3}{3}=0. \]
The manifold in $\jva[4](\R,\R^4)$ that those equations define is transverse to
the orbits in a neighborhood of any of its point, except where
$\det ( V(P))|_\mathcal{P} =\left(\x{2}{3}\right)^5\x{4}{4}$ vanishes.

Applying the replacement properties of normalized invariants defined by
this cross-section we immediately see that:
$$\kappa_{1} = \Lambda(\X[1],\X[2])=\inv \x{2}{3},\quad
\kappa_2 =  \Lambda(\X[2],\X[3])=-\inv\x{2}{3}\,\iota\x{3}{1},\,\dots $$
and
$$ \phi =  \det \left( \X[1],\X[2],\X[3],\X[4] \right) = \left(\inv
  \x{2}{3}\right)^2\,\inv \x{4}{4}.$$
From \Ref{mc} and  \Ref{serret} and the basis (\ref{smat},\ref{tmat})
of the Lie agebra,
a couple of steps of differential elimination handled by \emph{diffalg}
allow us to write the Maurer-Cartan form in terms of the symplectic curvatures:
\begin{equation} \label{previous:serret}
\rho^* \Omega \equiv -
 \left(\begin{array}{c|cccc}
0 & 0 & 0 & 0& 0\\  \hline \\
1& 0 & 0 & \displaystyle \frac{\kappa_2}{\kappa_1^2} & 1 \\ \\
0& 0 & 0 & 1 &  \displaystyle \tau \\ \\
0& \displaystyle -\kappa_1 & 0 & 0 & 0 \\ \\
0& 0 &  \displaystyle \frac{\phi}{\kappa_1^3} & 0 & 0
\end{array}\right) \dd t
\end{equation}
where
\begin{eqnarray*}
\tau & = &\frac{ -\inv \x{5}{2}+2\inv \x{2}{3}\inv \x{3}{1}-3\inv
  \x{4}{3}}{\inv \x{4}{4}}
\\ & =  &\frac{\kappa_1^4}{\phi^2}\,\kappa_4+
\frac{\kappa_1^3\kappa_2'}{\phi^2}\,\kappa_1'''
-\frac{(\kappa_1''-\kappa_2)\kappa_1^3}{\phi^2}\,\kappa_2''
+\frac{2\kappa_1'\kappa_1^2}{\phi^2}\,\phi'
-2\frac{\kappa_1^2}{\phi} (2\kappa_1''-\kappa_2)
 \\ &&
-\frac{\kappa_1^2}{\phi^2} (\kappa_2\kappa_1'\kappa_2'+\kappa_2\kappa_1''^2-2\kappa_1''\kappa_2^2
+2\kappa_2'^2\kappa_1-\kappa_1''\kappa_1'\kappa_2'
+\kappa_2^3) .
\end{eqnarray*}
We can similarly determine algorithmically all the normalized invariants in terms
of the symplectic curvatures. We organise the information in matrix form:
\begin{eqnarray} \label{previous:proj}
\begin{pmatrix}
\inv \X[0] \! &\! \inv\X[1] \! &\! \inv \X[2] \! &\! \inv\X[3] \! &\!  \inv\X[4]
\end{pmatrix}
\!\! & \!\! =\!\! &\!\!
\left( \begin{array}{ccccc}
0 &1& 0 &  -\frac{\kappa_2}{\kappa_1} & -\frac{\kappa_2'}{\kappa_1}
\\
0 & 0 & 0 &   -\kappa_1 &  -2\kappa_1'
\\
0 & 0 & \kappa_1 & \kappa_1' & \kappa_1''-\kappa_2
\\
0 & 0 & 0 &0 & \frac{\phi}{\kappa_1^2}
\end{array}\right).
\end{eqnarray}

Here the transversality condition is
$\inv (\det V(P))=-\kappa_1^3\,\phi\neq 0$ and both the moving frame
and the normalized invariants are well defined on the open set this
inequation defines.
Consider a linearly full curve $\gamma:I\rightarrow \R^4$,
its jets $j^{(k)}(\gamma):I\rightarrow \jva[k]$ and the functions
$\bar{\kappa}_i= \kappa_i \circ j^{(i+1)}(\gamma)$.
Assume $\bar{\kappa}_1$ does not vanish.
The curve $\rho( j^{(4)}(\gamma)) \star  j^{(4)}(\gamma)$ belongs to
the cross-section $\sva$ and is uniquely determined
by $\bar\kappa_1,\bar\kappa_2$ and $\bar\kappa_3$ as
can be seen from \Ref{previous:proj}.

Conversely, consider smooth functions
$\bar\kappa_1,\bar\kappa_2,\bar\kappa_3, \bar\kappa_4:I\rightarrow\R$
such that $\bar\kappa_1$ and
$\bar\phi=\bar\kappa_2^2-\bar\kappa_1\bar\kappa_3
   +\bar\kappa_1'\bar\kappa_2'-\bar\kappa_2\bar\kappa_1''$
do not vanish anywhere.
On one hand they define a curve
$\bar{\gamma}: I \rightarrow \sva\subset \jva[4]$  given by \Ref{previous:proj}.
On the other hand they define  a one-form $\bar\Omega$ on $I$
with values in $\ASp(4,\R)$ of the shape \Ref{previous:serret}.
By \theor{fundamental}, there is a curve $\bar{\rho}:I\rightarrow \ASp(4,\R)$ that is unique up to
an element of $\ASp(4,\R)$ such that $\bar{\rho}^*\Omega=\bar{\Omega}$.
Then $\bar{\rho}^{-1} \star \bar{\gamma}$ is the jet of a curve
$\gamma : I \rightarrow \R^4$ such that  the functions $\bar \kappa_i $ are its
symplectic curvatures.

\subsection{Minimal order section with  specialization property}
      \label{specialize} % \input{specialize}

The cross-section in  $\jva[4]$ of the previous paragraph is not transverse to the
jets of Lagrangian curves.
It was claimed in \cite{KOT,V} that Lagrangian curves
 require higher order frames.
Such is not the case. We exhibit two cross-sections in  $\jva[4]$
that are  transverse to full Lagrangian curves. The
first one is even of \emph{minimal order} \cite{olver07,hubert09}.

The equations
\[          \x{0}{1}=0,\;\x{0}{2}=0,\;\x{0}{3}=0,\;\x{0}{4}=0 \]
\[ \x{1}{1}=1,\;\x{1}{2}=0,\;\x{1}{3}=0,\;\x{1}{4}=0, \]
\[       \x{2}{1}=0,\;\x{2}{2}=1,  \;      \x{2}{4}=0\]
\[      \x{3}{1}=0,\;\x{3}{2}=0,\]
\[\x{4}{1}=0 .\]
define a cross-section in the neighborhood of any of its point, except
when $(\x{3}{4}\,\x{4}{2}\,\x{2}{3}+\x{4}{4}\,\x{3}{3} -\x{3}{4}\,\x{4}{3})\,\x{3}{4}=0$.

From the replacement properties of normalized invariants defined by
this cross-section we see that:
\[
\kappa_1 = \Lambda(\X[1],\X[2]) = \inv \x{2}{3},
\quad
\kappa_2 =  \Lambda(\X[2],\X[3]) = \inv \x{3}{4} ,
\quad
 \kappa_3 =  \Lambda(\X[4],\X[3]) = - \inv\x{3}{4}  \inv\x{4}{2}.
\]
Applying \Ref{serret} and \Ref{mc} with the basis
(\ref{smat},\ref{tmat}) of the Lie algebra,
and some differential elimination,
we algorithmically obtain
\begin{equation} \label{specialize:serret}
\rho^* \Omega \equiv  -\left(\begin{array}{c|cccc}
0 & 0 & 0 & 0 & 0 \\ \hline
1 &0 & \displaystyle\kappa_1\,\tau & -\tau
   &\displaystyle \frac{\kappa_1'\tau}{\kappa_2}
\\
0 & 1 & \displaystyle -\frac{\kappa_1\kappa_1' \tau}{\kappa_2}
& \displaystyle \frac{\kappa_1'  \tau}{\kappa_2}
&  \displaystyle -\frac{{\kappa_1'}^2\tau+\kappa_3}{\kappa_2^2}
 \\
0 &  \displaystyle \kappa_1 & 0 & 0 & -1
\\
0 & 0 &  \displaystyle \kappa_2+\kappa_1^2\tau
  & \displaystyle -\kappa_1\tau
 &\displaystyle \frac{\kappa_1\kappa_1' \tau}{\kappa_2}
\end{array}\right)\, \dd t ,
\end{equation}
where
\[\tau = \frac{\inv \x{3}{4}\,\inv \x{5}{1}}{\phi} =
\frac{\kappa_2^2}{\phi^2}\,\kappa_4
 -\frac{\kappa_2\kappa_2'}{\phi^2}\,\kappa_3'
+\frac{\kappa_2\kappa_3}{\phi^2}\, \kappa_2''
-\frac{\kappa_2\kappa_3^2}{\phi^2}.
\]
Furthermore, thanks to \Ref{rec}, we have:
\begin{eqnarray} \label{specialize:proj}
\begin{pmatrix}
\inv \X[0] \! &\!\inv\X[1] \! &\! \inv \X[2] \! &\! \inv\X[3] \! &\!  \inv\X[4]
\end{pmatrix}
\!\! & \!\! =\!\! &\!\!
\left( \begin{array}{ccccc}
0 & 1 & 0 & 0 &0
\\
0 & 0 & 1 & 0 & - \frac{\kappa_3}{\kappa_2}
\\
0 & 0 & \kappa_1 & \kappa_1' & \kappa_1''-\kappa_2
\\
0 & 0 & 0 & \kappa_2 & \kappa_2'
\end{array}\right).
\end{eqnarray}
This latter equality defines the moving frame explicitly since
\[\rho \cdot \begin{pmatrix}
           1 \! &\! 0 \! &\! 0 \! &\! 0\! &\! 0 \\
           \X[0] \! &\! \X[1] \! &\! \X[2] \! &\! \X[3] \! &\!  \X[4]
           \end{pmatrix}
= \begin{pmatrix}
1\! & \! 0\! & \! 0\! & \! 0\! & \!0 \\
\inv \X[0] \! &\! \inv\X[1]\! & \!\inv \X[2]\! & \!\inv\X[3]\! & \! \inv\X[4]
\end{pmatrix}
\]

Here the transversality condition is
$\inv\left(\det  V(P)\right)=-\kappa_2 \, \phi \neq 0$.
Consider a linearly full curve $\gamma:I\rightarrow \R^4$,
its jets $j^{(k)}(\gamma):I\rightarrow \jva[k]$ and the functions
$\bar{\kappa}_i= \kappa_i \circ j^{(i+1)}(\gamma)$.
Assume $\bar{\kappa}_2$ does not vanish.
The curve $\rho( j^{(4)}(\gamma)) \star  j^{(4)}(\gamma)$ belongs to
the cross-section $\sva$ and is uniquely determined
by $\bar\kappa_1,\bar\kappa_2$ and $\bar\kappa_3$ as
can be seen from \Ref{specialize:proj}.

Conversely, consider smooth functions
$\bar\kappa_1,\bar\kappa_2,\bar\kappa_3, \bar\kappa_4:I\rightarrow\R$
such that $\bar\kappa_2$ and
$\bar\phi=\bar\kappa_2^2-\bar\kappa_1\bar\kappa_3
   +\bar\kappa_1'\bar\kappa_2'-\bar\kappa_2\bar\kappa_1''$
do not vanish anywhere.
On one hand they define a curve
$\bar{\gamma}: I \rightarrow \sva\subset \jva[4]$ given by \Ref{specialize:proj}.
On the other hand they define  a one-form $\bar\Omega$ on $I$
with values in $\ASp(4,\R)$ of the shape \Ref{specialize:serret}.
By \theor{fundamental}, there is a curve $\bar{\rho}:I\rightarrow \ASp(4,\R)$ that is unique up to
an element of $\ASp(4,\R)$ such that $\bar{\rho}^*\Omega=\bar{\Omega}$.
Then $\bar{\rho}^{-1} \star \bar{\gamma}$ is the jet of a curve
$\gamma : I \rightarrow \R^4$ such that $\bar \kappa_i $ are its
symplectic curvatures.
We shall apply this construction to determine Lagrangian curves of
constant symplectic curvatures in Section~\ref{remarkable}.

\subsection{Section corresponding to a Gram-Schmidt process}
      \label{gram} %\input{gram}

There is a general method to obtain a moving frame $\rho:
\jva\rightarrow \ASp(2n,\R)$ that can be restricted to jets of
Lagrangian curves. It is obtained through a symplectic Gram-Schmidt
process for which we provide the case of $n=2$.
The underlying cross-section is not of minimal order, but the
 entailed Serret-Frenet matrix is sparser than the one proposed in previous paragraph.

The equations
\[          \x{0}{1}=0,\;\x{0}{2}=0,\;\x{0}{3}=0,\;\x{0}{4}=0 \]
\[ \x{1}{2}=1,\;\x{1}{4}=0, \]
\[       \x{2}{1}=1,\;\x{2}{2}=0,  \; \x{2}{3}=0,\;      \x{2}{4}=0,\]
\[      \x{3}{1}=0,\;\x{3}{2}=0, \;\x{3}{4}=0,\]
\[\x{4}{2}=0, \]
define a cross-section in the neighborhood of any of its point, except
when $\det V(P)|_{\mathcal{P}}=- \left(\x{3}{3}\right)^3\,\x{4}{4} $ vanishes.

From the replacement properties of normalized invariants defined by
this cross-section we see that:
\[
\kappa_1 = \Lambda(\X[1],\X[2]) = -\iota \x{1}{3},
\quad
\kappa_2 =  \Lambda(\X[2],\X[3]) = \iota \x{3}{3} ,
\quad
 \kappa_{3} =  \Lambda(\X[4],\X[3]) = - \iota\x{3}{3}  \iota\x{4}{1} .
\]
The transversality condition is
$\iota\left(\det V(P)\right)=\kappa_2^2 \, \phi \neq 0$.
Applying \Ref{serret} and \Ref{mc}, and some differential elimination,
we algorithmically obtain
\[
\rho^* \Omega \equiv \left(\begin{array}{c|cccc}
0 & 0 & 0 & 0 & 0 \\ \hline
 \ds-\frac{\kappa_1'}{\kappa_2} &0 & \ds -\frac{\phi}{\kappa_2^2}&\ds \frac{\kappa_3}{\kappa_2^2}   & 0
\\ \\
-1 & 0 & 0 & 0 &  \ds \tau
 \\ \\
 \ds \kappa_1 & \ds-\kappa_2 & 0 & 0 & 0
\\ \\
0 & 0 &  \ds \frac{\kappa_1\phi}{\kappa_2^2}
 &\ds \frac{\phi}{\kappa_2^2} & 0
\end{array}\right) \dd t ,
\]
where
$\tau$ is a fifth order invariant such that
\[ \kappa_4= \Lambda(\X[4],\X[5])= \frac{\phi^2}{\kappa_2^2}\,\tau
+ \frac{\kappa_2'\,\kappa_3'-\kappa_3\,\kappa_2''+\kappa_3^2}{\kappa_2
}. \]

The proposed cross-section was actually obtained
through a symplectic Gram-Schmidt process on
$\XX[{[+]}]=\left(\X[1],\X[2],\X[3],\X[4]\right)$
to define a moving frame.
Starting with the assumption that $\kappa_2=\Lambda(\X[2],\X[3])\neq 0$
we introduce the matrix\footnote{
We first  write $R$  above  in terms of the $\Lambda(\X[i],\X[j])$ so
has to visualize the Gram-Schmidt process. The expression of $R^{-1}$
in terms of the symplectic curvatures $\kappa_i$ is given afterwards. }
\[ R = \begin{pmatrix}
     0 & 1 & 0 & 0 \\ \\
  1 & \ds -\frac{\Lambda(\X[1],\X[3])}{\Lambda(\X[2],\X[3])} & 0 & \ds -\frac{\Lambda(\X[3],\X[4]) }{ \phi} \\ \\
0 & \ds \frac {\Lambda(\X[1],\X[2])}{\Lambda(\X[2],\X[3])} & \ds \frac{1}{ \Lambda(\X[2],\X[3]) } &\ds  \frac{ \Lambda(\X[2],\X[4])}{ \phi } \\ \\
0 & 0 & 0 & \ds -\frac{ \Lambda(\X[2],\X[3]) }{ \phi } \end{pmatrix}
\]
so that $Q= \XX[{[+]}] \, R\in\Sp(4,\R)$.

A moving frame $\rho : \jva[4] \rightarrow \ASp(4,\R)$ is then given by
$$\rho =\left( \begin{array}{cc}
1 & 0  \\ -Q^{-1} \X[0] & Q^{-1} \end{array} \right) \in \ASp(4,\R).$$
We can read the equations of the underlying  cross-section from
$R^{-1}$ since:
\begin{eqnarray*}
\begin{pmatrix}
1\! & \! 0\! & \! 0\! & \! 0\! & \!0 \\
\inv \X[0] \! &\! \inv\X[1]\! & \!\inv \X[2]\! & \!\inv\X[3]\! & \! \inv\X[4]
\end{pmatrix}
& = &
\rho %\left(\XX[{[4]}]\right)
\cdot \begin{pmatrix}
           1 \! &\! 0 \! &\! 0 \! &\! 0\! &\! 0 \\
           \X[0] \! &\! \X[1] \! &\! \X[2] \! &\! \X[3] \! &\!  \X[4]
           \end{pmatrix}
\\
 = \begin{pmatrix} 1 & 0 \\
 0 & R^{-1} \end{pmatrix}
& = &
\left( \begin{array}{c|cccc} 1 & 0 & 0 & 0 & 0 \\ \hline
0 &\ds \frac{\kappa_1'}{\kappa_2}& 1 & 0 & \ds -\frac{\kappa_3}{\kappa_2}
\\ \\
0 & 1 & 0 & 0 & 0
\\ \\
0 & -\kappa_1 & 0 & \kappa_2 & \kappa_2'
\\ \\
0 & 0 & 0 & 0& \ds -\frac{\phi}{\kappa_2}
\end{array}\right).
\end{eqnarray*}
The process can thus be generalised to determine appropriate
cross-sections for all dimensions.

% The fundamental theorem \cite[Theorem 5.2 and 6.1,]{sharpe97},
% \cite[Theorem 1.6.10]{ivey03} implies that,
% for an choice of functions $\kappa_1,\kappa_2,\kappa_3,\kappa_4$
% satisfying $\kappa_2\neq 0$ and
% $ \phi \neq 0$
% there exists a curve $\gamma$ in $\R^4$ having those Lagrangian
% curvatures. Furthermore for any two such cuves $\gamma_1$ and
% $\gamma_2$ there exists a group element
% $\lambda$ s.t. $\gamma_1 = \lambda \star \gamma_2$.

% We can thus specialize the Serret-Frenet equations for Lagrangian
% curves, when $\kappa_1 = 0$.
% The chosen cross-section is not of minimal order but results in
% a Serret-Frenet matrix that is sparser than with the previsous
% cross-section which was of minimal order and also had good
% spcialisation property for Lagrangian curves.

\section{Lagrangian curves with constant curvatures}
    \label{remarkable} %\input{remarkable}
\subsection{Serret-Frenet equations}

In Section~\ref{specialize}  (and \ref{gram}) we constructed a moving
frame that can be specialized to Lagrangian curves.
As discussed there,
given functions $\kappa_2,\kappa_3, \kappa_4:I\rightarrow\R$
such that $\kappa_2$  does not vanish
we obtain the unique curve, up to the action of $\ASp(4,\R)$,
with symplectic curvatures $\kappa_1=0$ and $\kappa_2,\kappa_3,
\kappa_4$ as the first column of $\rho^{-1}$ where $\rho: I \rightarrow \ASp(4,\R) $ is a matrix solution of
\[ \rho^* \Omega = - A(\kappa)\,\dd t , \quad \rho(0)\in \ASp(4,\R) \]
where, in accordance with (\ref{specialize:serret})
\[ A(\kappa) = \left(\begin{array}{c|cccc}
0 & 0 & 0 & 0 & 0 \\ \hline
1 &0 & 0 & -\tau  &0
\\
0 & 1 & 0 &0 &  \ds -\frac{\kappa_3}{\kappa_2^2}
 \\
0 & 0 & 0 & 0 & -1
\\
0 & 0 &  \kappa_2 & 0 &0
\end{array}\right) \in \ASpe(4,\R) \]
and
\[\tau =\frac{1}{\kappa_2^2}\,\kappa_4
 -\frac{\kappa_2'}{\kappa_2^3}\,\kappa_3'
+\frac{\kappa_3}{\kappa_2^3}\, \kappa_2''
-\frac{\kappa_3^2}{\kappa_2^3}.
\]
Equivalently we can directly consider the equation for $\tilde{\rho}=\rho^{-1}$
which is the analogue of the  Serret-Frenet  equation
\[ \dt{ \tilde{\rho} } = \tilde{\rho} \, A(\kappa) .\]
In terms of a frame $(p,E_1,E_2,E_1,E_2)$ those are the linear differential equations
\[ p'= E_1,\, E_1'=E_2,\, E_2'=\kappa_2 E_3,\, E_2'=-\tau E_1, \, E_3'=\frac{\kappa_3}{\kappa_2^2}E_2-E_3. \]

\subsection{Classification} \label{classif}

In this section we classify the Lagrangian curves parameterized by arc-length
that have constant symplectic curvatures.
 In other words we investigate the curves with symplectic curvatures
$\kappa_1=0,\kappa_2=1$ and $\kappa_3,\kappa_4$ constant.
For each pair of constants $(\kappa_3,\kappa_4)$
any such  curve is given by the first column of $B \, e^{As}$
where $B\in \ASp(4,\R)$ and
\[ A= \left(\begin{array}{c|cccc}
0 & 0 & 0 & 0 & 0 \\ \hline
1 &0 & 0 & \kappa_3^2 - \kappa_4  &0
\\
0 & 1 & 0 &0 &  \ds -\kappa_3
 \\
0 & 0 & 0 & 0 & -1
\\
0 & 0 &  1 & 0 &0
\end{array}\right).
\]
They can thus be classified according to the
roots of the characteristic polynomial
$\lambda \, \pi(\lambda)=\lambda (\lambda^4 +\kappa_3 \lambda^2
+\kappa_3^2-\kappa_4)$ of $A$.
In each case we provide a curve in the congruence class. It is
described in terms of some $\mu,\nu\in \R$ that are taken such that $\mu>\nu>0$.

\begin{itemize}
\item[I:] $(\kappa_3^2-\kappa_4)(4\kappa_4 - 3\kappa_3^2)\neq
  0$. $\pi$ has four distinct roots.

\begin{itemize}
\item[I.1:] $4\kappa_4<3\kappa_3^2$.
 The roots of $\pi$  are $\{\pm \mu\pm i\nu\}$.
\[{\gamma(s)}   =
\left(\begin{array}{c}
\ds\frac{(\mu^2-3\nu^2)\mu\cos (\nu s)\sinh(\mu s)}{\mu^2+\nu^2}
\\\ds
-\frac{1}{2}\frac{\sin (\nu s)\sinh(\mu s) }{(\mu^2+\nu^2)^2\mu\nu}
+\frac{1}{2}\frac{(\mu^2-3\nu^2)\cos (\nu s)\cosh(\mu s) }{(\nu^2-3\mu^2)(\mu^2+\nu^2)^2\nu^2}
\\\ds
- \frac{1}{2}\frac{\sin (\nu s)\sinh(\mu s) }{(\mu^2+\nu^2)^2\mu\nu}
+\frac{1}{2}\frac{(3\mu^2-\nu^2)\cos (\nu s)\cosh(\mu s) }{(\mu^2-3\nu^2) (\mu^2+\nu^2)^2\mu^2}
\\\ds
\frac{(\nu^2-3\mu^2)\nu\sin (\nu s)\cosh(\mu s)}{\mu^2+\nu^2}
\end{array}\right)\]
\item[I.2:]  $4\kappa_4>3\kappa_3^2$.
The squares of the roots of $\pi$  are real.
\begin{itemize}
\item[I.2.a:] $\pm \sqrt{4\kappa_4 - 3\kappa_3^2}>\kappa_3$.
The roots of $\pi$  are $\{\pm\mu,\pm \nu\}$.
\[ \tr{\gamma(s)}= \begin{pmatrix}
\ds \frac{\sinh (\nu s)}{\sqrt{\mu^2-\nu^2} \nu^{\frac{3}{2}} } ,&
\ds \frac{\cosh (\mu s)}{\sqrt{\mu^2-\nu^2} \mu^{\frac{3}{2}}}, &
\ds\frac{\cosh (\nu s)}{\sqrt{\mu^2-\nu^2} \nu^{\frac{3}{2}} },&
\ds \frac{\sinh (\mu s)}{\sqrt{\mu^2-\nu^2} \mu^{\frac{3}{2}}}
\end{pmatrix}\]
\item[I.2.b:] $\pm \sqrt{4\kappa_4 - 3\kappa_3^2}<\kappa_3$.
 The roots of $\pi$
 are $\{\pm i\mu,\pm i\nu\}$.
\[ \tr{\gamma(s)}= \begin{pmatrix}
\ds\frac{\sin (\nu s)}{\sqrt{\mu^2-\nu^2} \nu^{\frac{3}{2}} },&
\ds\frac{\cos (\mu s)}{\sqrt{\mu^2-\nu^2} \mu^{\frac{3}{2}}} ,&
\ds\frac{\cos (\nu s)}{\sqrt{\mu^2-\nu^2} \nu^{\frac{3}{2}} },&
\ds\frac{\sin (\mu s)}{\sqrt{\mu^2-\nu^2} \mu^{\frac{3}{2}}}
\end{pmatrix}\]
\item[I.2.c:] $- \sqrt{4\kappa_4 - 3\kappa_3^2}<\kappa_3< \sqrt{4\kappa_4 - 3\kappa_3^2}$.
The roots of $\pi$  are $\{\pm i\mu,\pm\nu\}$.
\[ \tr{\gamma(s)}= \begin{pmatrix}
\ds-\frac{\sinh (\nu s)}{\sqrt{\mu^2+\nu^2} \nu^{\frac{3}{2}}} ,&
\ds\frac{\cos (\mu s)}{\sqrt{\mu^2+\nu^2} \mu^{\frac{3}{2}}}, &
\ds\frac{\cosh (\nu s)}{\sqrt{\mu^2+\nu^2} \nu^{\frac{3}{2}}} ,&
\ds\frac{\sin (\mu s)}{\sqrt{\mu^2+\nu^2} \mu^{\frac{3}{2}}}
\end{pmatrix}\]
\end{itemize}
\end{itemize}

\item[II:] $\kappa_4=\kappa_3^2$, $\kappa_3\neq 0$.
Then $\pi(\lambda)=\lambda^2(\lambda^2+\kappa_3)$.

\begin{itemize}
\item[II.1:]  $\kappa_3<0$.
The non-zero roots of $\pi$ are $\{\mu, -\mu\}$, where
$\mu=\sqrt{-\kappa_3}>0$.
\[ \tr{\gamma(s)}= \begin{pmatrix}
\ds s,&
\ds\frac{\cosh (\mu s)}{\mu^{\frac{5}{2}}}, &
\ds \frac{s^2}{2 \mu^{2}} ,&
\ds\frac{\sin (\mu s)}{\mu^{\frac{5}{2}}} \end{pmatrix}\]
\item[II.2:] $\kappa_3>0$.
The non-zero roots of $\pi$ are $\{i\mu, -i\mu\}$, where
$\mu=\sqrt{\kappa_3}>0$.
\[ \tr{\gamma(s)}= \begin{pmatrix}
\ds s,&
\ds\frac{\cos (\mu s)}{\mu^{\frac{5}{2}}}, &
\ds- \frac{s^2}{2 \mu^{2}} ,&
\ds\frac{\sin (\mu s)}{\mu^{\frac{5}{2}}} \end{pmatrix}\]
\end{itemize}

\item[III:] $4\kappa_4=3\kappa_3^2$ and $\kappa_3\neq 0$.
Then $\pi(\lambda)=(\lambda^2+\frac{1}{2}\kappa_3)^2$.

\begin{itemize}
\item[III.1:] $\kappa_3>0$.
The roots of $\pi$ are $\{i\mu, -i\mu\}$ where $\mu=\sqrt{\frac{\kappa_3}{2}}>0$.
\[ \tr{\gamma(s)}= \begin{pmatrix}
\ds\frac{s\,\cos (\mu s)}{\mu^{2}}, &
\ds \frac{s\,\sin (\mu s)}{\mu} + \frac{3 \cos (\mu s)}{2 \mu^{2}} ,&
\ds-\frac{1}{2}\frac{\cos (\mu s)}{\mu^{2}}, &
\ds -\frac{1}{2}\frac{\sin (\mu s)}{\mu^{3}}
 \end{pmatrix}\]
\item[III.2:] $\kappa_3<0$.
The roots of $\pi$ are $\{\mu, -\mu\}$ where $\mu=\sqrt{-\frac{\kappa_3}{2}}>0$.
\[ \tr{\gamma(s)}= \begin{pmatrix}
\ds\frac{s\,\cosh (\mu s)}{\mu^{2}}, &
\ds \frac{s\,\sinh (\mu s)}{\mu} - \frac{3 \cos (\mu s)}{2 \mu^{2}} ,&
\ds-\frac{1}{2}\frac{\cosh (\mu s)}{\mu^{2}},&
\ds \frac{1}{2}\frac{\sinh (\mu s)}{\mu^{3}}
 \end{pmatrix}\]
\end{itemize}

\item[IV:]  $\kappa_3=\kappa_4=0$.  The roots of $\pi$ are zero.
 \[\tr{\gamma(s)}= \begin{pmatrix}
\ds \frac{t}{\sqrt{24}}, & \ds\frac{t^2}{\sqrt{12}},   & -\ds\frac{t^4}{\sqrt{24}}, & \ds \frac{t^3}{\sqrt{12}}
\end{pmatrix}\]

\end{itemize}

%\begin{remark}
All the curves cataloged in the
  previous classification are orbits of one-parameter subgroups of the
  affine symplectic group. For instance, the curve $\mathrm{I.2.b}$ is
  the orbit passing through the point
$$p_{\mu,\nu}=\tr{\left(0,(\mu^3(\mu^2-\nu^2))^{-1/2},(\nu^3(\mu^2-\nu^2))^{-1/2},0\right)}$$
of the $1$-parameter subgroup of $\mathcal{S}(4,\R)$ generated by
$$\mathrm{X}_{\mu,\nu}=\left(
                         \begin{array}{cc}
                           0 & 0 \\
                           0 & J_{\mu,\nu} \\
                         \end{array}
                       \right)\in \mathfrak{s}(4,\R),$$
where
$$J_{\mu,\nu}=\mu(C^2_2-B^2_2)+\nu(B^1_1-C^1_1)\in \mathfrak{sp}(4,\R).$$
In the previous formula, the $4\times 4$ matrices $B^i_j$ and $C^i_j$
denote the generators of $\mathfrak{sp}(4,\R)$ defined in
Subsection~1.1.
 The symplectic curvatures $\kappa_3,\kappa_4$ and the constants
$\mu,\nu$ are related by
$$\kappa_3=\mu^2+\nu^2,\quad \kappa_4=\mu^4+\mu^2\nu^2+\nu^4.$$
Also in the other cases it is possible to write explicitly the
$1$-parameter subgroups that generate the curves.

We would like to highlight here a
fact about the curves of type $\mathrm{I.2.b}$~: their trajectories
are the integral curves of a linear flow on the homogenous Lagrangian
torus $\mathrm{T}^{2}_{\mu,\nu}\subset \R^4$. Such a torus is the
orbit through the point $p_{\mu,\nu}\in \R^4$ of the maximal Abelian
subgroup $\mathrm{T}^2\cong \mathrm{SO}(2)\times \mathrm{SO}(2) $ of
$\mathcal{S}(4,\R)$ consisting of all
$\mathrm{X}(\theta_1,\theta_2)\in \mathcal{S}(4,\R)$ of the following
form
\[ \mathrm{X}(\theta_1,\theta_2) = \left( \begin{array}{c|cccc} 1 & 0 & 0 & 0 & 0 \\ \hline
       0 & \cos(\theta_1) & 0 & -\sin(\theta_1) & 0 \\ 0  & 0 & \cos(\theta_2) & 0 & -\sin(\theta_2) \\ 0  & \sin(\theta_1) & 0 & \cos(\theta_1) &  0
      \\ 0  & 0 & \sin(\theta_2) & 0 &
      \cos(\theta_2) \end{array}\right),\quad \theta_1\theta_2\in \R.
\]
This fact has an immediate geometric consequence~: the trajectories of
the curves of type $\mathrm{I.2.b}$ are either simple closed curves or
else are dense in the $2$-dimensional torus
$\mathrm{T}^{2}_{\mu,\nu}$. Another noteworthy observation is that the
$2$-dimensional tori $\mathrm{T}^{2}_{\mu,\nu}$ are the fibers of the
moment map $\mathrm{m}:\R^4\to \mathfrak{t}^*$ induced by the
Hamiltonian action of the maximal compact Abelian subgroup
$\mathrm{T}^2$. The symbol $\mathfrak{t}^*$ stands for the dual of the
Lie algebra of $\mathrm{T}^2$.
%\end{remark}

\subsection{Closed curves}

Closed trajectories occur only for curves of Type I.2.b  when
 $\frac{\mu}{\nu} \in \Q$.
The conditions on $\kappa_3,\kappa_4$ can be simplified to
\[ \kappa_3^2 > \kappa_4 >\frac{3}{4} \kappa_3^2
\quad \hbox{ and }\quad
\kappa_3 > 0 .\]
Then
\[ \mu = \sqrt{\frac{\kappa_3+\sqrt{4\kappa_4-3\kappa^2}}{2}}
\quad\hbox{ and }\quad
 \nu = \sqrt{\frac{\kappa_3-\sqrt{4\kappa_4-3\kappa^2}}{2}}. \]
The fact that $\frac{\mu}{\nu}=\frac{m}{n}$, for $m,n\in \N$ can be
rewritten as follows.

\begin{propy}
A Lagrangian curve $\gamma$ % parameterized by arclength
with constant symplectic curvatures
$\kappa_3,\kappa_4$ is closed if and only if
$\kappa_3>0$, $\kappa_3^2>\kappa_4>\frac{3}{4}\kappa_3^2$ and there exist $m,n\in\N$
\[ \frac{\kappa_3^2}{(m^2+n^2)^2}=\frac{\kappa_4}{m^4+m^2m^2+n^4} .\]
Taking $m$ and $n$ relatively prime, $\gamma$ is a torus knot of type $(m,n)$.
Its symplectic length is $\frac{2\pi  }{\nu \,n}\hbox{lcm}(m,n)$.
\end{propy}

%\textcolor{blue}{Tracking down Emilio's different notations, I have:
%\[ k_1=k^1=a=p^2+q^2= -\kappa_3, \quad k^2=k_2=b=p^2q^2=\kappa_3^2-\kappa_4 .\]
%Unfortunately there seems to be a mismatch.}

\begin{figure}[ht]
\begin{center}
\includegraphics[height=6cm,width=12cm]{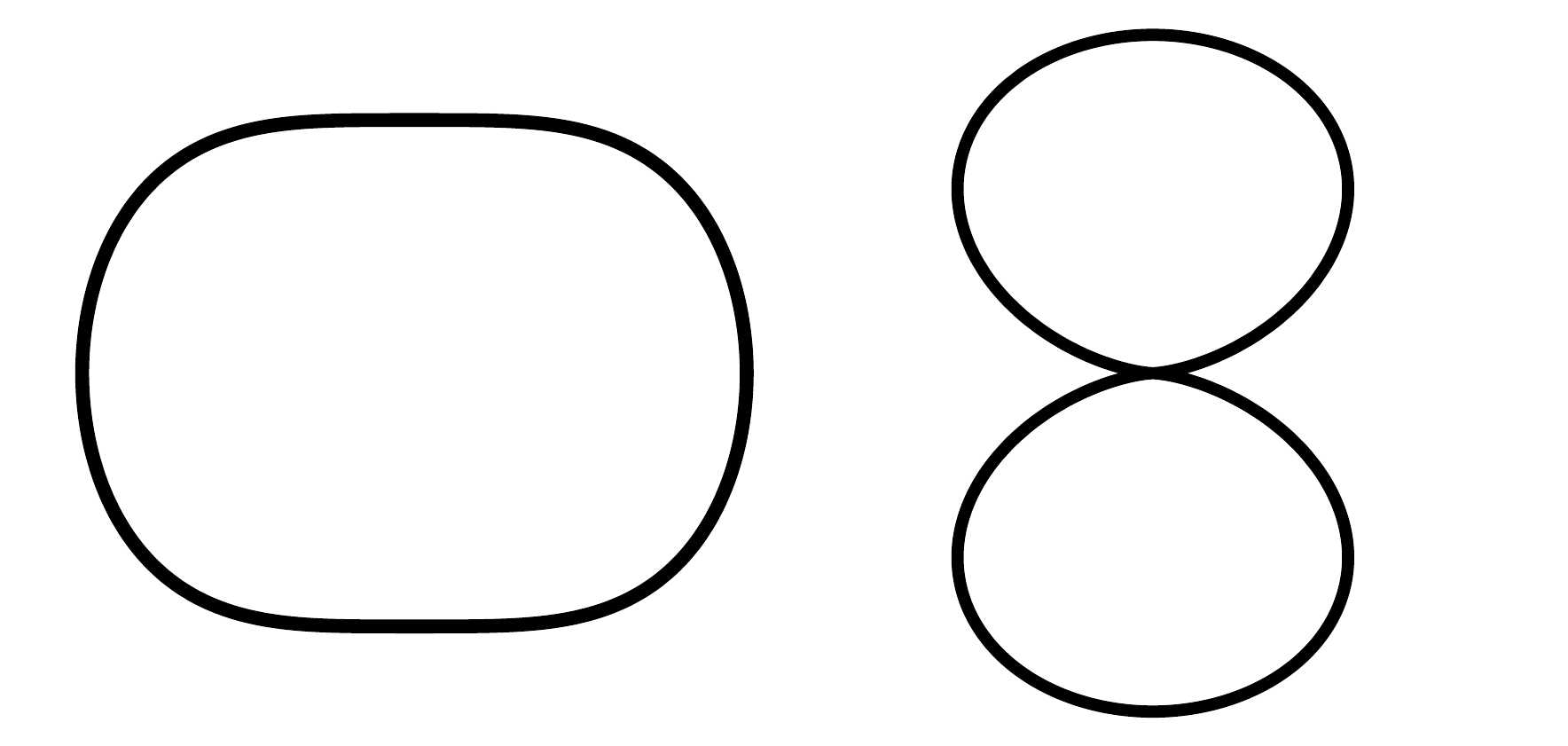}
\caption{Phase portraits of a Lagrangian curve with $\kappa_3=2.5$ and $\kappa_4=5.6875$.}\label{FIG2}
\end{center}
\end{figure}

\begin{figure}[ht]
\begin{center}
\includegraphics[height=6cm,width=12cm]{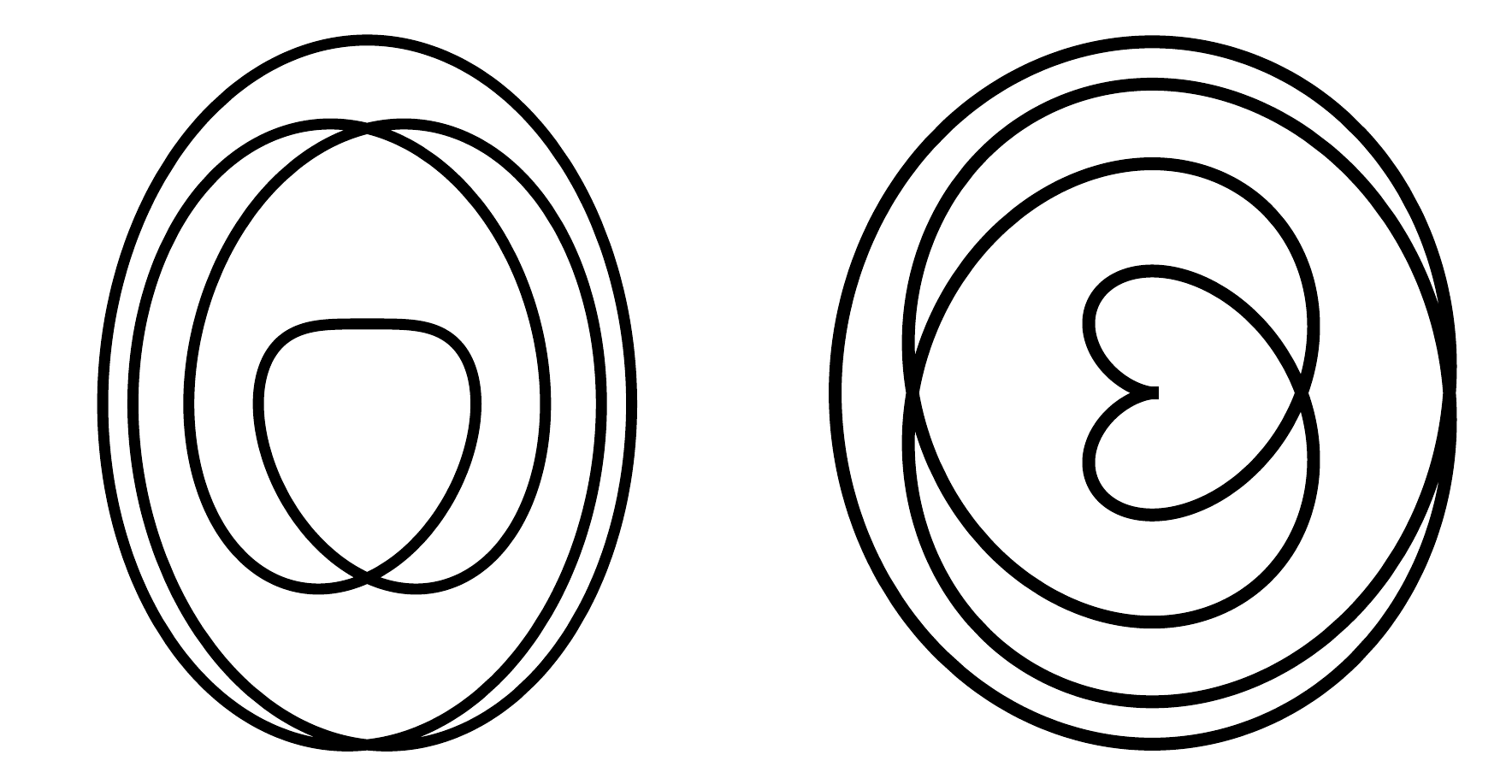}
\caption{Phase portraits of a Lagrangian curve with $\kappa_3=1.64$ and $\kappa_4=2.0496$.}\label{FIG3}
\end{center}
\end{figure}

The concept of second-order deformations of curves
 introduced in Section~\ref{dpairs} can be illustrated on those curves.
If  $(\mathbf{a},\mathbf{b})$ is the pair  of
 phase portraits of a
Lagrangian curve $\gamma_{(\mathbf{a},\mathbf{b})}$, the curves
  $\mathbf{a}$ and $\mathbf{b}$  are
 second-order deformations of each other with respect to the special affine group (d-pairs for short).
Given an element $g\in \mathcal{S}(4,\R)$, $g\star \gamma_{(\mathbf{a},\mathbf{b})}$ is
another Lagrangian curve whose phase portraits
$(\tilde{\mathbf{a}},\tilde{\mathbf{b}})$  make up a new d-pair.
In this way  $\mathcal{S}(4,\R)$ acts on the space of all d-pairs.
Such an action is global in nature and is not originated by a
pseudo-group of transformations of the plane. Its effects can be
rather unpredictable, as we wish to show in  Figure~\ref{FIG2} and  Figure~\ref{FIG3}.

The classification shows that any closed Lagrangian curve
with constant symplectic curvatures is congruent to a curve whose
phase portrait consists of a pair of circles. That such a  pair of
circles are  second-order deformations of
each other with respect to the special affine group is rather easy to
visualise.
Other elements in the congruence class provide more surprising d-pairs.
Figure~\ref{FIG2} reproduces the phase portraits of a
closed Lagrangian curve with constant symplectic curvatures
$\kappa_3=\frac{5}{2}$, $\kappa_4 =\frac{91}{16}$,
which correspond to the data $\nu=\frac{1}{2}$,
$m=3$ and $n = 1$. The symplectic length of such a curve is
$\ell\approx 12.5664$.
Figure~\ref{FIG3} provides the phase portraits of a closed Lagrangian curve with symplectic curvatures
$\kappa_3=1.64$ and $\kappa_4=2.0496$, which correspond to the values $\nu=0.8$,
$m = 5$, $n = 4$. The symplectic length is $\ell \approx 31.4259$.

\subsection{Lagrangian tori}

We recall that an immersed surface $f:S\to \R^4$ is Lagrangian if its
tangent planes are Lagrangian. This notion plays a central role in the
Hamilton-Jacobi theory. Among all Lagrangian surfaces, the Lagrangian
tori have a particular significance because they arise in a natural
way as the fibers of the momentum map of a Liouville-integrable
Hamiltonian system. We now exhibit a procedure to construct families
of Lagrangian tori starting from closed Lagrangian curves with constant curvatures.

Let $\gamma:\R\to \R^4$ be a closed Lagrangian curve with constant curvatures and
symplectic length $\ell_{\gamma}$.
We assume that $\gamma$ is parameterized by the symplectic
arc-length.   Therefore $\kappa_3>0$ and $\kappa_3^2-\kappa_4>0$.
We fix a real constant $h$ and
take any closed regular curve, parameterized by the Euclidean
arc-length, $\alpha :\vartheta\in \R\to \R^3$
whose trajectory lies in the quadric $\mathcal{Q}_h\subset \R^3$ defined by the equation
$x^2+\kappa_3 y^2-z^2=h$. By construction, $\mathcal{Q}_0$ is a
quadratic cone, $\mathcal{Q}_h$, $h<0$, is a two-sheet hyperboloid and
$\mathcal{Q}_h$, $h>0$, is a one-sheet hyperboloid. Let assume that
the trajectory belongs either to $\mathcal{Q}_h^+=\{(x,y,z)\in
\mathcal{Q}_h | z>0\}$ or else to $\mathcal{Q}_h^-=\{(x,y,z)\in
\mathcal{Q}_h | z<0\}$. If $h<0$ the condition is automatically
fulfilled. When, $h=0$, we are actually imposing that the vertex of the
cone does not belongs to the trajectory of the curve.
Finally, when  $h>0$, this imposes that
the trajectory lies either in the upper or in the lower parts of the
one sheet hyperboloid.

Denoting by $(E_1,...,E_4):\R\to \mathrm{Sp}(4,\R)$ the symplectic Frenet frame along $\gamma$ we define
$$f(s,\vartheta)=\gamma(s)+x(\vartheta)E_2(s)+\frac{1}{\kappa_3^2-\kappa_4}(\sqrt{\kappa_3^2-\kappa_4}z(\vartheta)-1)E_3(s)+y(\vartheta)E_4(s)\in
\R^4,$$
where $x(\vartheta),y(\vartheta)$ and $z(\vartheta)$ are the
components of $\alpha$. Geometrically, our surface is a sort of
molding surface with directrix curve $\gamma$ and profile $\alpha$.

 By
construction, $f$ is a doubly periodic map and its lattice of periods
$L_{(\gamma,\alpha)}$ is generated by $(\ell_{\gamma},0)$ and
$(0,\ell_{\alpha})$, where $\ell_{\alpha}$ is the Euclidean length of
$\alpha$. Thus, $f$ induces a smooth map
$\tilde{f}:\R^2/L_{(\gamma,\alpha)}\to \R^4$. Using the Serret-Frenet
equations satisfied by the symplectic frame we obtain
$$(\partial_s f)|_{(s,\vartheta)}=\sqrt{\kappa_3^2-\kappa_4}z(\vartheta)E_1(s)-\kappa_3y(\vartheta)E_2(s)-y(\vartheta)E_3(s)+x(\vartheta)E_4(s).$$
and
$$(\partial_{\vartheta} f)|_{(s,\vartheta)} = \dot{x}(\vartheta)E_2(s)+\frac{1}{\sqrt{\kappa_3^2-\kappa_4}}\dot{z}(\vartheta)E_3(s)+\dot{y}(\vartheta)E_4(s),$$
where $\dot{x},\dot{y}$ and $\dot{z}$ are the derivatives with respect
to the parameter $\vartheta$. Since $z(\vartheta)\neq 0$
and
$\dot{x}(\vartheta)^2+\dot{y}(\vartheta)^2+\dot{z}(\vartheta)^2>0$,
for every $\vartheta$, the maps $f$ and $\tilde{f}$ are smooth
immersions. Moreover, the two previous equations also imply
$$\begin{array}{ll}
\Lambda(\partial_s f,\partial_{\vartheta} f)|_{(s,\vartheta)}
&
\ds =-x(\vartheta)\dot{x}(\vartheta)-\kappa_3y(\vartheta)\dot{y}(\vartheta)+z(\vartheta)\dot{z}(\vartheta)
\\
& \ds
=-\frac{1}{2}\frac{d}{d\vartheta}\left(x^2+\kappa_3y^2-z^2\right)|_{\vartheta}=0.
\end{array}$$
This shows that $f$ and $\tilde{f}$ are Lagrangian immersions.

\section{Lagrangian geodesics}
    \label{geodesics} % \input{geodesics}

In this section we introduce a concept of geodesics for Lagrangian
curves. Interestingly, they form a subset of the Lagrangian curves with
constant curvatures.

\begin{defy}
Let $\mathcal{L}$ be the space of linearly full Lagrangian curves in $\R^4$.
By a {\it Lagrangian variation} of $\gamma \in \mathcal{L}$ we mean a mapping $\Gamma : I \times (-\epsilon, \epsilon)
\to \R^4$ such that $\gamma_u : = \Gamma(-,u) : I \to \R^4$ is a linearly full Lagrangian curve, $\forall u \in (-\epsilon,
\epsilon)$ and $\Gamma(t,0)=\gamma(t)$, for every $t\in I$. The infinitesimal variation of $\Gamma$ is the vector field along $\gamma$ defined by
\[\mathfrak{v}:t\in I\to \partial_u\Gamma|_{(t,0)}\in \R^4 . \]
If $\mathfrak{v}$ vanishes outside a closed interval, then $\Gamma$ is said to be compactly supported.\end{defy}

\begin{defy}
A curve $\gamma\in \mathcal{L}$ is said to be a
\textit{Lagrangian  geodesic}
if it is a critical point of the {\it symplectic arclength functional}
\begin{equation}\label{action}
\ell : \gamma \in \mathcal{L} \mapsto
  \int \sigma_{\gamma}\in \R,
    \end{equation}
when one considers compactly supported variations.
\end{defy}

Accordingly, a curve $\gamma \in \mathcal{L}$ is a Lagrangian geodesic
if, for every compactly supported variation $\Gamma$, we have that
$$ \left. \frac{d}{du}   \left( \int_{K} \sigma_u \right)\right|_{u=0}= 0,$$
where $K$ is the smallest closed interval which contains the support
of the infinitesimal variation $\mathfrak{v}$.

\begin{defy}
Let $\gamma : I \to \R^4$ be a Lagrangian curve. A vector field
$\mathfrak{v}:I\to \R^4$ along $\gamma$ is said to be an
\emph{infinitesimal Lagrangian variation} if
$\mathfrak{v}|_t= \partial_u\Gamma|_{(t,0)}$ for some Lagrangian
variation
$\Gamma:I\times
(-\epsilon,\epsilon) \to \R^4$ of $\gamma$.
\end{defy}

The set of all infinitesimal Lagrangian
variations of $\gamma$ should be thought of as the tangent space  $T_{\gamma}(\mathcal{L})$ of
$\mathcal{L}$ at $\gamma$.

\begin{theoy}
A linearly full Lagrangian curve $\gamma$ parameterized by the symplectic arc-length is a
    Lagrangian geodesic if and only if $\kappa_3$ is a constant and $\kappa_4=\kappa_3^2$.
\end{theoy}

Consequently geodesics are of  Type II or IV in the classification
of Section~\ref{classif}.

For simplification we introduce $k_1=-\kappa_3$ and
$k_2=\kappa_3^2-\kappa_4$ as the symplectic curvatures of a Lagrangian
curve when parameterized by symplectic arc-length.
As preliminaries to the  proof
we first characterize the infinitesimal variations of a linearly full
Lagrangian curve.
We then derive the Euler-Lagrange equations of the symplectic arc-length.
\begin{lemy}
Let $\gamma:I\to \R^4$ be a linearly full Lagrangian curve parameterized by symplectic arc-length. A vector field
$$\mathfrak{v} : s\in I \to (v_1(s),...,v_4(s))\in \R^4$$
along $\gamma$ is an infinitesimal Lagrangian variation if and only if
\(\label{INFV}v_2=\frac{1}{2}(v_3'' - 3v_4').\)
\end{lemy}

\begin{proof}{Suppose that $\mathfrak{v}$ is the infinitesimal variation induced by
the Lagrangian variation $\Gamma(s,u)$ of $\gamma$.
Possibly restricting the interval $(-\epsilon,\epsilon)$,
we can assume that all the curves $\gamma_u : s\in I\to \Gamma(s,u)$
are linearly-full.
Let
$$\mathcal{F}:I\times (-\epsilon,\epsilon)\to \mathcal{S}(4,\R)$$
be the Frenet frame along the variation, i.e. the map which
associate to each $(s,u)$ the symplectic Frenet frame of $\gamma_u$ at
the point $\gamma_u(s)$.  If we set $\Omega = \mathcal{F}^{-1}d\mathcal{F}$, then
$$\Omega = \mathcal{K}(s,u)ds + \mathcal{P}(s,u)du,$$
for $\mathfrak{s}(4,\R)$-valued functions
$$\mathcal{K}=\left(
      \begin{array}{ccccc}
        0 & 0 & 0 & 0 & 0 \\
        1 & 1 & 0 & \mathcal{K}_2 & 0 \\
        0 & 0 & 0 & 0 & \mathcal{K}_1 \\
        0 & 0 & 0 & 0 & -1 \\
        0 & 0 & \mathcal{U} & 0 & 0 \\
      \end{array}
    \right),\quad \mathcal{P} = \left(
                        \begin{array}{ccccc}
                          0 & 0 & 0 & 0 & 0 \\
                          \mathcal{V}_1 & \mathcal{A}_1 & \mathcal{A}_2 & \mathcal{B}_1 & \mathcal{B}_2 \\
                          \mathcal{V}_2 & \mathcal{A}_3 & \mathcal{A}_4 & \mathcal{B}_2 & \mathcal{B}_3 \\
                          \mathcal{V}_3 & \mathcal{C}_1 & \mathcal{C}_2 & -\mathcal{A}_1  & -\mathcal{A}_2 \\
                          \mathcal{V}_4 & \mathcal{C}_2 & \mathcal{C}_3 & -\mathcal{A}_3 & - \mathcal{A}_4  \\
                        \end{array}
                      \right),
$$
 such that
\begin{equation}\label{CData}\mathcal{U}(s,0)=1,\quad
\mathcal{V}_j(s,0)=v_j(s),\quad \mathcal{K}_1(s,0)=k_1(s),\quad
\mathcal{K}_2(s,0)=k_2(s),
\end{equation}
where $k_1(s)$ and $k_2(s)$ are the symplectic curvatures at $\gamma(s)$.
By construction, $\Omega$ satisfies
\begin{equation}\label{MCC1} d\Omega=-\Omega\wedge \Omega.
\end{equation}
Equation (\ref{MCC1}) can be rewritten in the form
\begin{equation}\label{MCC2} \frac{\partial \mathcal{K}}{\partial u}- \frac{\partial \mathcal{P}}{\partial s}=[\mathcal{K},\mathcal{P}].
\end{equation}
In turn, (\ref{MCC2}) is equivalent to the following system of equations
\begin{equation}\label{MCC3}
 \left\{\begin{array}{llll}
\mathcal{V}_2=\mathcal{U}^{-1}(\mathcal{V}_3^{(2,0)}-3\mathcal{V}_4^{(1,0)})/2,\\
\mathcal{A}_J=\sum_{m=1}^{4}\sum_{h=1}^{4}\mathcal{U}^{-r_1(J,m,h)}A^m_{J,h}(j_s(\mathcal{U}),j_s(\mathcal{K}_1),j_s(\mathcal{K}_2))\mathcal{V}_m^{(h,0)},\\
\mathcal{C}_p=\sum_{m=1}^{4}\sum_{h=1}^{3} \mathcal{U}^{-r_2(p,m,h)}C^m_{p,h}(j_s(\mathcal{U}),j_s(\mathcal{K}_1),j_s(\mathcal{K}_2))\mathcal{V}_m^{(h,0)},\\
\mathcal{B}_p=\sum_{m=1}^{4}\sum_{h=1}^{5} \mathcal{U}^{-r_3(p,m,h)}B^m_{p,h}(j_s(\mathcal{U}),j_s(\mathcal{K}_1),j_s(\mathcal{K}_2))\mathcal{V}_m^{(h,0)},\\
\mathcal{U}^{(0,1)}=\sum_{m=1}^{4}\sum_{h=1}^{4} \mathcal{U}^{-\tilde{r}(m,h)}U^m_{h}(j_s(\mathcal{U}),j_s(\mathcal{K}_1),j_s(\mathcal{K}_2))\mathcal{V}_m^{(h,0)},\\
\mathcal{K}_1^{(0,1)}=\sum_{m=1}^{4}\sum_{h=1}^{6} \mathcal{U}^{-\tilde{r}_1(m,h)}K^m_{1,h}(j_s(\mathcal{U}),j_s(\mathcal{K}_1),j_s(\mathcal{K}_2))\mathcal{V}_m^{(h,0)},\\
\mathcal{K}_2^{(0,1)}=\sum_{m=1}^{4}\sum_{h=1}^{5} \mathcal{U}^{-\tilde{r}_2(m,h)}K^m_{2,h}(j_s(\mathcal{U}),j_s(\mathcal{K}_1),j_s(\mathcal{K}_2))\mathcal{V}_m^{(h,0)},
\end{array}\right.
\end{equation}
where $A^m_{J,h},C^m_{p,h},B^m_{p,h},U^m_{p,h},K^m_{1,h},K^m_{2,h}$ are suitable polynomial functions, $j_s(f)$ and }is the jet with respect to the variable $s$ of a function $f(s,u)$, $f^{(h,k)}=\partial^h_s \partial^k_u f$ and the exponents $r_1(J,m,h)$, $r_2(p,m,h)$, $r_3(p,m,h)$, $\tilde{r}(m,h)$, $\tilde{r}_1(m,h)$, $\tilde{r}_2(m,h)$ are non-negative integers. Note that all these quantities can be calculated explicitly. However, the only explicit formula that will be useful in the following is the derivative of the function $\mathcal{U}$ with respect to the parameter $u$, which can be written in the form

\begin{equation}\label{MCC4}
\begin{array}{llll} \mathcal{U}^{(0,1)} = & \mathcal{U}^{(1,0)} \mathcal{V}_1 +5\mathcal{U} \mathcal{V}_1^{(1,0)}+2\mathcal{U} \mathcal{K}_2\mathcal{V}_3+\\
&\left(\mathcal{U}\mathcal{K}_1 + \frac{5}{2}\left(\frac{\mathcal{U}^{(1,0)} }{\mathcal{U}}\right)^2-\frac{3}{2}\frac{\mathcal{U}^{(2,0)}}{\mathcal{U}}\right)\mathcal{V}_3^{(2,0)}
-\frac{5}{2}\frac{\mathcal{U}^{(1,0)}}{\mathcal{U}}\mathcal{V}_3^{(3,0)}+\\
&2 \mathcal{V}_3^{(4,0)}+\left(\mathcal{K}_1 \mathcal{U}^{(1,0)}+3 \mathcal{U}\mathcal{K}_1^{(1,0)}\right)\mathcal{V}_4+\\
&\left(2 \mathcal{U}
\mathcal{K}_1 -\frac{15}{2}\left(\frac{\mathcal{U}^{(1,0)}}{\mathcal{U}}\right)^2+\frac{9}{2}\frac{\mathcal{U}^{(2,0)}}{ \mathcal{U}}\right)\mathcal{V}_4^{(1,0)}+\\
&\frac{15}{2}\frac{
\mathcal{U}^{(1,0)} }{\mathcal{U}}\mathcal{V}_4^{(2,0)}-5 \mathcal{V}_4^{(3,0)}
\end{array}.
\end{equation}
Keeping in mind that $\mathcal{V}_m(s,0)=v_m(0)$, $m=1,...,4$, and
$\mathcal{U}(s,0)=1$, the first formula in (\ref{MCC3}) implies
$v_2=\frac{1}{2}(v_3'' - 3v_4')$. Conversely, let us consider four
real-valued smooth functions $v_1,..,v_4$ defined on the open interval
$I$ such that $v_2=\frac{1}{2}(v_3'' - 3v_4')$. We let $A_J$,
$J=1,..,4$, $B_p$ and $C_p$, $p=1,2,3$, be the functions form $I$ to
$\R$ obtained by placing $\mathcal{U}=1$, $\mathcal{K}_1=k_1$ and
$\mathcal{K}_2=k_2$
in the right hand side of the second, third and fourth equations of
(\ref{MCC3}). Similarly, we let $\dot{u}$,$\dot{k}_1$ and $\dot{k}_2$
be the function defined by putting $\mathcal{U}=1$,
$\mathcal{K}_1=k_1$ and $\mathcal{K}_2=k_2$ in the right hand side of
the last three equations of (\ref{MCC3}). Next we consider the
functions
$K,P,Q:I\to \mathfrak{s}(4,\R)$ defined by
\begin{equation}\label{MCC4.I}
K=\left(
      \begin{array}{ccccc}
        0 & 0 & 0 & 0 & 0 \\
        1 & 1 & 0 & k_2 & 0 \\
        0 & 0 & 0 & 0 & k_1 \\
        0 & 0 & 0 & 0 & -1 \\
        0 & 0 & 1 & 0 & 0 \\
      \end{array}
    \right),\quad  P = \left(
                        \begin{array}{ccccc}
                          0 & 0 & 0 & 0 & 0 \\
                          v_1 & A_1 & A_2 & B_1 & B_2 \\
                          v_2 & A_3 & A_4 & B_2 & B_3 \\
                          v_3 & C_1 & C_2 & -A_1  & -A_2 \\
                          v_4 & C_2 & C_3 & -A_3 & - A_4  \\
                        \end{array}
                      \right),
\end{equation}
and by
$$Q=\left(
      \begin{array}{ccccc}
        0 & 0 & 0 & 0 & 0 \\
        0 & 0 & 0 & \dot{k}_2 & 0 \\
        0 & 0 & 0 & 0 & \dot{k}_1 \\
        0 & 0 & 0 & 0 & 0 \\
        0 & 0 & \dot{u} & 0 & 0 \\
      \end{array}
    \right).
$$
By construction, $P$ is a solution of the o.d.e
$$\frac{dP}{ds}=Q-[K,P].$$
Subsequently, we let $\mathcal{K}$ be the map
\begin{equation}\label{MCC5}\mathcal{K}:(s,t)\in I\times \R\to K(s)+tQ(s),\end{equation}
 and we let $\mathcal{P}:I\times \R\to \mathfrak{s}(4,\R)$ be the solution of the equation
\begin{equation}\label{MCC6}\frac{\partial \mathcal{P}}{\partial s}=Q-[K,\mathcal{P}]-t[Q,\mathcal{P}],\quad \mathcal{P}(s,0)=P(s_0),\end{equation}
where $s_0$ is an element of $I$. The maps $P$ and $\mathcal{P}(-,0)$ are solutions of the same o.d.e with the same Cauchy data $P(s_0)=\mathcal{P}(s_0,0)$. From this we infer that $P(s)=\mathcal{P}(s,0)$, for every $s\in I$. Then, we consider the $\mathfrak{s}(4,\R)$-valued $1$-form
\begin{equation}\label{MCC7}\Omega = \mathcal{K}ds+\mathcal{P}du\in \Omega^1(I\times \R)\otimes \mathfrak{s}(4,\R).
\end{equation}
From (\ref{MCC5}) and (\ref{MCC6}) it follows that
\begin{equation}\label{MCC8}d\Omega = -\Omega\wedge \Omega.\end{equation}
This implies the existence of a smooth map
\begin{equation}\label{MCC9}\mathcal{F}=(\Gamma,\mathbf{E}):I\times \R\to \mathcal{S}(4,\R)\cong \R^4\times \mathrm{Sp}(4,\R)\end{equation}
such that
\begin{equation}\label{MCC10}\mathcal{F}^{-1}d\mathcal{F}=\Omega,\quad \mathcal{F}(s_0,0)=(\gamma(s_0),\mathbf{E}_0),\end{equation}
where $\mathbf{E}_0\in \mathrm{Sp}(4,\R)$ is the symplectic Frenet frame $\mathbf{E}_{\gamma}$ of $\gamma$ evaluated at $s_0$. From (\ref{MCC4.I}), (\ref{MCC5}), (\ref{MCC7}) and (\ref{MCC10}) we see that the map
$$(\tilde{\gamma},\tilde{\mathbf{E}}):s\in I \to \mathcal{F}(s,0)=(\Gamma(s,0),\mathbf{E}(s,0))\in \mathcal{S}(4,\R)$$
satisfies
\begin{equation}\label{MCC11}\tilde{\gamma}'=\mathbf{E}_1,\quad \tilde{\mathbf{E}}' = \tilde{\mathbf{E}}\cdot K,\end{equation}
with the Cauchy data $\tilde{\gamma}(s_0)=\gamma(s_0)$ and $\tilde{\mathbf{E}}(s_0)=\mathbf{F}_{\gamma}(s_0)$. On the other hand, (\ref{MCC11}) is just the Frenet system satisfied by the canonical moving frame along $\gamma$. From this we deduce that $$\Gamma(s,0)=\gamma(s),\quad \mathbf{E}(s,0)=\mathbf{F}_{\gamma}(s),\quad \forall s\in I.$$
This means that $\Gamma$ is a Lagrangian variation  of $\gamma$. Furthermore, using once again (\ref{MCC4.I}), (\ref{MCC5}), (\ref{MCC7}) and (\ref{MCC10}) we see that $\frac{\partial \Gamma}{\partial u}|_{(s,0)}=\mathfrak{v}$. We have thus constructed a Lagrangian variation of $\gamma$ that has $\mathfrak{v}$ as its infinitesimal variation. This concludes the proof of the Lemma.
\end{proof}

\begin{lemy}
Let $\Gamma:I\times (-\epsilon,\epsilon)\to \R^4$ be a compactly
supported variation of a linearly full Lagrangian curve $\gamma :I\to
\R^4$   parameterized by the symplectic arc-length and let
$\mathfrak{v}:I\to \R^4$ be its infinitesimal variation. Then,
\begin{equation}\label{EL1}
\frac{d}{du}\left(\int_K \sigma_u\right)|_0 = \int_K \left((2k_2(s)+k_1''(s))v_3(s)+k_1'(s)v_4(s)\right)ds,
\end{equation}
where $v_1,...,v_4$ are the components of $\mathfrak{v}$, $k_1, k_2$ are the symplectic curvatures of $\gamma$, $\sigma_u$ is the symplectic arc-element of the Lagrangian curves $\gamma_u$ swept out by
the variation and $K\subset I$ is the smallest closed interval
containing the support of $\mathfrak{v}$.
\end{lemy}

\begin{proof}
We maintain the same notations that have been used in the proof of the previous lemma. Then we have
$$\frac{d}{du}\left(\int_K \sigma_u\right)|_0  = \frac{d}{du}\left(\int_K \mathcal{U}(s,u)ds\right)|_{u=0} =
\int_K\mathcal{U}^{(0,1)}(s,0)ds.
$$
From (\ref{MCC4}) and keeping in mind that
$$\mathcal{U}(s,0)=1,\quad \mathcal{K}_1(s,0)=k_1(s),\quad \mathcal{K}_2(s,0)=k_2(s)$$
we have
$$
\int_K\mathcal{U}^{(0,1)}(s,0)ds = \int_K\left(2k_2 v_3 + k_1 v_3'' + 2v_3^{(4)} + 3k_1' v_4 + 2k_1 v_4'-5v_4^{(3)} \right)ds.
$$
Taking into account that the supports of the functions $v_1,..,v_4$ are subsets of $K$ and integrating by parts, we obtain
$$
\begin{array}{llll}
&\quad \int_K\left(2k_2 v_3 + k_1 v_3'' + 2v_3^{(4)} + 3k_1' v_4 + 2k_1 v_4'-5v_4^{(3)} \right)ds =\\
&=\int_K\left(2k_2v_3+k_1v_3''+3k_1'v_4+2k_1v_4'\right)ds = \\
&=\int_K\left(2k_2v_3+(k_1v_3')'-k_1'v_3'+2(k_1v_4)'+k_1'v_4\right)ds = \\
&= \int_K\left((2k_2+k_1'')v_3+k_1v_4-(k_1'v_3)'\right)ds = \int_K\left((2k_2+k_1'')v_3+k_1v_4\right)ds
\end{array}.
$$
This yields the required result.
\end{proof}

\begin{proof} We are now in a position to prove the proposition. From the
  Lemma above we see that if $k_1$ is constant and $k_2=0$, then
  $\gamma$ is automatically a critical point of the symplectic
  arc-length functional with respect to compactly supported
  variations. Conversely, suppose that $\gamma$ is a Lagrangian
  geodesic of $\R^4$. Take any compactly supported function $v_4:I\to
  \R$, set $v_1=v_3=0$ and $v_2=-3v_4'/2$. Then, from the first Lemma
  we know that there is a Lagrangian variation of $\gamma$ whose
  infinitesimal variation is given by
  $\mathfrak{v}=(0,v_2,0,v_4)$. Using the second Lemma we obtain
$$0=\frac{d}{du}\left(\int_K \sigma_u\right)|_0 = \int_K k_1 v_4ds.$$
Since the function $v_4$ is arbitrary (provided with compact support),
this implies that $k_1'=0$, i.e. $k_1$ is a constant. Next, take any
compactly supported smooth function $v_3:I\to \R$ and set
$\mathfrak{v}=(0,v_3''/2,v_3,0)$. Thus, using again the first Lemma, we
deduce the existence of a compactly supported Lagrangian variation of
$\gamma$ having $\mathfrak{v}$ as its infinitesimal
variation. Therefore, using again the second Lemma we obtain
$$0=\frac{d}{du}\left(\int_K \sigma_u\right)|_0 = 2\int_K k_2 v_3ds.$$
On the other hand, $v_3$ can be any compactly supported smooth
function. Therefore $k_2$ vanishes identically. We have thus proved
the theorem.
\end{proof}

The result could have been inferred using a more conceptual framework,
based on the Griffiths' approach to the calculus of variations in one
independent variable \cite{Gr2, GM,MN1,MN2}.
However, this point of view require
a considerable amount of preliminary work, such as the construction of
an appropriate exterior differential system on the configuration space
$\mathcal{S}(4,\R)\times \R^2$ and the computation of the so called
Euler-Lagrange system, whose integral curves give back the critical
points of the functional.
Similarly, the Euler-Lagrange operator may be obtained in the
framework of \cite{kogan03}.

\bibliographystyle{plain}

\end{document}